\documentclass{article}
\usepackage{geometry}
\usepackage{amsmath}
\usepackage{latexsym}
\usepackage{amsthm}
\usepackage{amsfonts}
\usepackage{amssymb}
\usepackage{tikz}
\usepackage{tikz-cd}
\usepackage{gensymb}
\usepackage{cancel}
\usepackage{setspace}
\usepackage{mathtools}
\usepackage{framed}
\usepackage{tgbonum}
\usepackage{hyperref}
\usepackage{graphicx}
\usepackage{changepage}
\usetikzlibrary{decorations.pathreplacing,angles,quotes}
\date{}
\author{Adam Afandi}
\title{Stationary Descendents and the Discriminant Modular Form}

\newtheorem{Definition}{Definition}
\newtheorem{Proposition}{Proposition}
\newtheorem{Lemma}{Lemma}
\newtheorem{Theorem}{Theorem}
\newtheorem{Corollary}{Corollary}
\newtheorem{Example}{Example}

\newtheorem{Open Problem}{Open Problem}
\newtheorem{Conjecture}{Conjecture}

\newcommand{\mbar}{\overline{\mathcal{M}}}

\begin{document}
\maketitle

\begin{abstract}

The generating functions of stationary descendent Gromov-Witten invariants of an elliptic curve are known to be  Fourier expansions of quasimodular forms. When one restricts to the subspace of forms of a fixed weight $k$, there is an abundance of linear relations among these generating functions. This naturally leads one to study the resulting linear matroid, which we refer to as the \emph{descendent matroid of weight $k$}. In the case of weight 12, we use this matroid to compute and organize all of the ways to express the discriminant modular form in terms of stationary descendents. As a consequence, we find a closed-form expression of Ramanujan tau values in terms of Gromov-Witten invariants of an elliptic curve. All computations were aided with the use of Sage, and the classes and functions written in Sage are discussed in the appendix.

\end{abstract}

\tableofcontents

\pagebreak

\section{Introduction}
 
Let $z \in \mathcal{H}$ be a coordinate in the upper half plane. Recall the {\bf{discriminant modular form}} $\Delta$, given by

\begin{equation*}
\Delta(z) := q\prod_{n \geq 1}(1 - q^n)^{24} := \sum_{n \geq 1}\tau(n)q^n
\end{equation*}

\noindent where $q = e^{2 \pi i z}$. The function $\tau: \mathbb{N} \rightarrow \mathbb{Z}$ is often referred to as the {\bf{Ramanujan tau function}}, in honor of Ramanujan's work on $\tau(n)$ \cite{ramanujan1916certain}. What makes $\tau(n)$ an intriguing function are its remarkable arithmetic properties. For example, we have the following properties, once conjectured by Ramanujan himself, and then proven by Mordell \cite{mordell1917mr}  and Deligne \cite{deligne1974conjecture}:

\begin{enumerate}
\item{$(m, n) = 1 \implies \tau(mn) = \tau(m)\tau(n)$} 
\item{$p \ \text{prime}, \ r > 0 \implies \tau(p^{r+1}) = \tau(p)\tau(p^r) - p^{11}\tau(p^{r-1})$}
\item{$p \ \text{prime} \implies |\tau(p)| \leq 2p^{\frac{11}{2}}$}
\end{enumerate}

\noindent One of the longstanding problems concerning the Ramanujan tau function is proving that $\tau(n) \not= 0$ for all $n \in \mathbb{N}$. This problem is often referred to as \emph{Lehmer's conjecture}, as Lehmer was the first to posit its validity \cite{lehmer1947vanishing}. The computational evidence corroborating Lehmer's conjecture is substantial, see for example (\cite{serre2012course}, pg. 98).

In order to gain more insight as to why Lehmer's conjecture might be true, it's worthwhile to find new ways to interpret what the Ramanujan tau function is actually computing. In \cite{niebur1975formula}, Niebur found a closed form expression for $\tau(n)$ in terms of the sum-of-divisors function,

\begin{equation*}
\tau(n) = n^4\sigma(n) - 24\sum_{i = 1}^{n - 1}i^2(35i^2 - 52in + 18n^2)\sigma(i)\sigma(n - i)
\end{equation*}

\noindent In \cite{garvan2018}, Garvan and Schlosser found a combinatorial expression for $\tau(n)$ in terms of vector partition functions. In the same spirit of trying to find different interpretations of $\tau(n)$, in this paper, we systematically compute all of the ways one can express $\tau(n)$ in terms of disconnected Gromov-Witten invariants of an elliptic curve. The motivation of our investigation is to provide more tools and perspectives to attack Lehmer's conjecture, and to gain a better understanding of how Gromov-Witten invariants generate the ring of quasimodular forms.

First, let's recall why such a computation is even possible. The function $\Delta$ is an important example of a {\bf{modular form}}, that is, an element in the ring $\text{M} := \mathbb{Q}[E_4, E_6]$, where $E_k$ is the (non-normalized) Eisenstein series of weight $k \in 2\mathbb{Z}$, defined by 

\begin{equation}\label{eisenstein}
E_k := -\frac{B_k}{2k} + \sum_{n \geq 1}\sigma_{k - 1}(n)q^n
\end{equation}

\noindent where $B_k$ is the $k^{th}$ Bernoulli number, and $\sigma_{k - 1}(n) := \sum_{d | n}d^{k - 1}$. It's a nice exercise to show that

\begin{equation*}
\Delta = 8000E_4^3 - 147E_6^2
\end{equation*}

\noindent In the literature, one often sees $\Delta$ expressed in the normalized Eisenstein series, in which the Eisenstein series has constant term $1$. In this case, the expression for $\Delta$ reads $(E_4^3 - E_6^2)/1728$. In this paper, we choose to use the non-normalized Eisenstein series, as it aligns better with the results that are cited.

We say that the modular form $E_4^iE_6^j \in \text{M}$, has \emph{weight} $4i + 6j$. This weight function induces a homogeneous grading on $\text{M} := \bigoplus M_k$, where $M_k := \{\text{forms of weight $k$}\}$. We can use Equation \ref{eisenstein} to define the Eisenstein series $E_2$. However, $E_2$ is not a modular form, but is rather an example of a {\bf{quasimodular form}}. The ring of quasimodular forms $\text{QM} := \mathbb{Q}[E_2, E_4, E_6]$ was first introduced and studied by Zagier and Kaneko in \cite{kaneko1995generalized}, and the theory was further elaborated in \cite{zagier2008elliptic}. Similar to the case of modular forms, we say a monomial $E_2^aE_4^bE_6^c \in \text{QM}$ has weight $2a + 4b + 6c$, which induces the grading $\text{QM} = \bigoplus \text{QM}_k$, where $\text{QM}_k := \{\text{forms of weight $k$}\}$.

Let $\mbar^{\bullet}_{g, n}(E, d)$ be the moduli stack of possibly disconnected stable maps of degree $d$ from $n$-pointed genus $g$ curves to a smooth elliptic curve $E$. This space has virtual dimension $2g - 2 + n$. It comes equipped with $n$ evaluation maps $\text{ev}_i: \mbar^{\bullet}_{g, n}(E, d) \rightarrow E$. Furthermore, we let $\psi_i := c_1(\mathbb{L}_i) \in H^2(\mbar^{\bullet}_{g, n}(E, d))$, where $\mathbb{L}_i$ is the $i^{th}$ universal cotangent line bundle. We call the $\psi$-classes {\bf{descendent classes}}. The {\bf{disconnected stationary descendent Gromov-Witten invariants}} of $E$ are intersection numbers of the form

\begin{equation*}
\left<\tau_{k_1}\ldots\tau_{k_n}\right>^{\bullet E}_{d} := \int_{\left[.\mbar_{g, n}^{\bullet}(E, d) \right]^{\text{vir}}}\psi_1^{k_1}\text{ev}_1^*([pt.])\ldots\psi_n^{k_n}\text{ev}_n^*([pt.])
\end{equation*}

\noindent The starting point of our investigation begins with a result of Okounkov and Pandharipande:

\begin{Theorem}\label{OP}[\cite{okounkov2006gromov}]
Let $k_1, \ldots, k_n \geq 0$ be integers. Then the generating function

\begin{equation*}
\left<\tau_{k_1}\ldots\tau_{k_n}\right> := \prod_{n \geq 1}(1 - q^n)\sum_{d \geq 0}\left<\tau_{k_1}\ldots\tau_{k_n}\right>^{\bullet E}_d q^d
\end{equation*}

\noindent is a quasimodular form of weight $k = \sum_{i = 1}^n(k_i + 2)$.

\end{Theorem}

\noindent Throughout this paper, we refer to the generating functions $\left<\tau_{k_1}\ldots\tau_{k_n}\right>$ simply as {\bf{stationary descendents of weight $k = \sum (k_i + 2)$}}.

Theorem \ref{OP} tells us that we have a $\mathbb{Q}$-linear map

\begin{equation*}
\text{GW}_k: \bigoplus_{\sum(k_i + 2) = k} \mathbb{Q}\cdot\left<\tau_{k_1}\ldots\tau_{k_n}\right> \rightarrow \text{QM}_k
\end{equation*}

\noindent In order to compute $\Delta$ in terms of stationary descendents of weight 12, we need to understand the linear map $\text{GW}_{12}$. However, it turns out that trying to understand $\text{GW}_k$ for all weights $k$ leads to some interesting combinatorics. The reason is as follows. If we want to understand $\text{GW}_k$, we need to compute its kernel and image. But as $k$ grows larger, the amount of linear-algebraic data that determines $\text{GW}_k$ becomes formidable. Fortunately, there already exists combinatorial machinery that helps one think clearly about this data, namely, the theory of {\bf{matroids}}. 

Matroids are combinatorial objects that abstract the notion of linear independence. They were first introduced by Whitney \cite{whitney1992abstract}, and form a pillar of modern research in combinatorics \cite{oxley2011matroid}. For a nice survey on how matroids have steered mathematical progress up to the present time, we recommend the ICM article of Ardila \cite{ardila2021geometry}.

Each linear map $\text{GW}_k$ is completely determined by the matroid $\mathcal{M}_k$, whose ground set is the set of stationary descendents of weight $k$, and whose bases are the collections of stationary descendents that form a basis of $\text{QM}_k$. We refer to this matroid as the {\bf{descendent matroid of weight $k$}}. In this paper, we initiate the study of the matroids $\mathcal{M}_k$.

\subsection{Results}

\begin{Theorem}
For $k = 2, 4, 6$, $\mathcal{M}_k$ is the uniform matroid $U_{1, 1}, U_{2, 2}, U_{3, 4}$, respectively.
\end{Theorem}

\noindent The weight $k = 8$ is the first instance in which $\mathcal{M}_k$ is not uniform. 

\begin{Theorem}\label{weighttwelve}
Let $\mathcal{M}$ be the rank-$4$ matroid with ground set $E = (abcdefg)$, whose bases are given by $\mathcal{B} = \{\text{$4$-element subsets of $E$}\}\setminus\{abcd\}$. Then $\mathcal{M}_8 \cong \mathcal{M}$.

\end{Theorem}

\noindent For $k \geq 10$, the matroid $\mathcal{M}_k$ has much more combinatorial complexity. However, for $k = 10, 12$, we show that $\mathcal{M}_k$ can be restricted to a uniform matroid in a natural way. For any $k$, let $S^k$ be the {\bf{positive stationary descendents}} of weight $k$ i.e. the quasimodular forms $\left<\tau_{k_1}\ldots\tau_{k_n}\right>$ where $k_i > 0$ for all $i$.    

\begin{Theorem}
Let $S^{k}$ be the positive stationary descendents of weight $k$. Then $\mathcal{M}_{10} | S^{10} = U_{5, 5}$ and $\mathcal{M}_{12} | S^{12} = U_{7, 9}$. 
\end{Theorem}

\noindent For $k > 12$, the restricted matroid $\mathcal{M}_{k} | S^{k}$ is not uniform. However, we conjecture that it is always of full rank (see Conjecture \ref{fullrank}), and that it always has a certain type of uniform matroid as a minor (see Conjecture \ref{minor}). \\

\noindent The 36 bases of $\mathcal{M}_{12} | S^{12} = U_{7, 9}$ correspond to 36 unique ways of expressing the discriminant modular form $\Delta$ as a $\mathbb{Q}$-linear sum of stationary descendents of weight $12$. The sums are collected in Tables \ref{partone}, \ref{parttwo}, and \ref{partthree} in the Appendix. Translating this into a statement about the Ramanujan tau function, we have the following corollary:

\begin{Corollary}
Let $\tau(d)$ be the Ramanujan tau function evaluated at the integer $d > 0$. For all bases $B \in \mathcal{B}^{12} | S^{12}$, there exists a unique vector $(a_{b})_{b \in B} \in \mathbb{Q}^7$ such that

\begin{equation*}
\tau(d) = \sum_{3d_1^2 - d_1 + 2d_2 = 2d \atop b \in B}(-1)^{d_1}a_b\left<\tau_{b}\right>^{\bullet E}_{d_2}
\end{equation*}
\end{Corollary}

\noindent By the above corollary, Lehmer's conjecture is equivalent to the following: for $d \in \mathbb{N}$, the degree-$d$ coefficient of all 36 power series collected in Tables \ref{partone}, \ref{parttwo}, and \ref{partthree} never vanishes. Whether this is a step closer to resolving Lehmer's conjecture remains to be seen. However, our investigation provides new relations for Gromov-Witten invariants of an elliptic curve $E$ by using the arithmetic properties of $\tau(n)$. Furthermore, the introduction of the descendent matroid of weight $k$ may be of independent interest.

\subsection{Acknowledgements}

This work was supported by the Cluster of Excellence \emph{Mathematics Münster} and the CRC 1442 \emph{Geometry: Deformations and Rigidity}.``Funded by the Deutsche Forschungsgemeinschaft (DFG, German Research Foundation) - Project-I 427320536 - SFB 1442, as well as under Germany's Excellence Strategy EXC 2044390685587, Mathematics Münster: Dynamics - Geometry - Structure"

\section{Stationary Descendents and Quasimodular Forms}

In this section, we introduce the ring of quasimodular forms $\text{QM}$, and recall the relevant results in \cite{okounkov2006gromov} that allow us to go back and forth between Gromov-Witten invariants of an elliptic curve $E$ and quasimodular forms. \\

\subsection{The Ring of Quasimodular Forms}

Let's begin with the classical theory of modular forms on the full modular group. See \cite{diamond2005first} for a standard reference.

\begin{Definition}
Let $f: \mathcal{H} \rightarrow \mathbb{C}$ be a holomorphic function on the upper half plane $\mathcal{H}$ that is holomorphic at $\infty$. For any even integer $k$, we say $f$ is a {\bf{modular form}} of weight $k$ if, for all $\gamma = \begin{bmatrix} a & b \\ c & d \end{bmatrix} \in \text{SL}_2(\mathbb{Z})$ and for all $z \in \mathcal{H}$,

\begin{equation*}
f(\gamma(z)) = (cz + d)^kf(z)
\end{equation*}

\end{Definition}

\noindent Every modular form $f$ can be expanded into its Fourier coefficients $a_n(f)$,

\begin{equation*}
f(z) = \sum_{n = 0}^\infty a_n(f)e^{2\pi i z n} = \sum_{n = 0}^\infty a_n(f) q^n, \hspace{0.5cm} q := 2\pi i z
\end{equation*} 

\noindent When we write $f(q)$, we are referring to the Fourier expansion of $f$, and will otherwise write $f(z)$. The {\bf{Eisenstein series}} of weight $k$, denoted $E_k$, is defined to be

\begin{equation*}
E_k := -\frac{B_k}{2k} + \sum_{n \geq 1}\sigma_{k - 1}(n)q^n
\end{equation*}

\noindent It turns out that $E_4$ and $E_6$ are modular forms of weight $4$ and $6$, respectively, and that the ring of modular forms $\text{M}$ is freely generated by $E_4$ and $E_6$ i.e. $\text{M} = \mathbb{Q}[E_4, E_6]$. The weight of the modular form $E_4^iE_6^j$ is $4i + 6j$, and this induces a grading on $\text{M} = \bigoplus M_k$, where $M_k := \{\text{forms of weight $k$}\}$. In this paper, we are interested in finding different ways to express the {\bf{discriminant modular form}}.

\begin{Definition}
The {\bf{discriminant modular form}} $\Delta$ is the unique \emph{cusp form} (i.e. modular form with vanishing constant term) of weight $12$ with Fourier coefficient $a_1 = 1$. Specifically,

\begin{equation*}
\Delta := q\prod_{n \geq 1}(1 - q^n)^24 = \sum_{n \geq 1}\tau(n) q^n
\end{equation*}

\end{Definition}

\noindent One would hope that $E_2$ is a `modular form of weight 2', but this is not the case, since

\begin{equation}\label{eisensteinweighttwo}
E_2(\gamma z) = (cz + d)^2E_2(z) + \frac{ic}{4\pi}(cz + d)
\end{equation}

\noindent for all $\gamma \in \text{SL}_2(\mathbb{Z})$ and $z \in \mathcal{H}$ (see \cite{zagier2008elliptic}, pg. 19). To allow $E_2$ to mingle with $\text{M}$ in a consistent way, we relax the condition of modularity.

\begin{Definition}[\cite{royer2012quasimodular}, Definition 3.2]
Let $f: \mathcal{H} \rightarrow \mathbb{C}$ be a holomorphic function. We say $f$ is a {\bf{quasimodular form}} of weight $k$ and depth $s \geq 0$ if there exists holomorphic functions $f_0, \ldots, f_s : \mathcal{H} \rightarrow \mathbb{C}$ with $f_s \not= 0$, such that, for all $\gamma = \begin{bmatrix} a & b \\ c & d \end{bmatrix} \in \text{SL}_2(\mathbb{Z})$ and $z \in \mathcal{H}$,

\begin{equation*}
(cz + d)^{-k}f(\gamma z) = \sum_{j = 0}^s f_j(z)\left( \frac{c}{cz + d} \right)^j
\end{equation*}

\end{Definition}

\noindent The Eisenstein series $E_2$ is a quasimodular form of weight $2$ and depth $1$. Indeed, if we let $f_0 = E_2$ and $f_1 = \frac{i}{4\pi}$, we get back the transformation of Equation \ref{eisensteinweighttwo}. Furthermore, every modular form of weight $k$ is also a quasimodular form of weight $k$ and depth $0$.

The ring of quasimodular forms $\text{QM}$ is freely generated by $E_2, E_4$, and $E_6$, so that $\text{QM} = \mathbb{Q}[E_2, E_4, E_6]$. The monomial $E_2^aE_4^bE_6^c$ is a quasimodular form of weight $2a + 4b + 6c$, and depth $a$. As in the case of $\text{M}$, $\text{QM} = \bigoplus \text{QM}_k$ has a natural grading by weight. Intuitively, one should think of the depth of a quasimodular form $f$ as a measure of how far $f$ is from being modular. Indeed, any quasimodular form $f \in \text{QM}$ of weight $k$ and depth $s$ can be uniquely written as

\begin{equation*}
f = \sum_{j = 0}^s m_jE_2^j
\end{equation*} 

\noindent where $m_j$ is a modular form of weight $k - 2j$. The higher the depth, the more nonmodular terms appear in the sum above.

\subsection{Stationary Descendents of Weight $k$}

Let $g$ and $n$ be nonnegative integers satisfying $2g - 2 + n > 0$. Let $X$ be a smooth projective variety, and let $d \in H_2(X)$ be a curve class. If $(C, p_1, \ldots, p_n)$ is a nodal curve of genus $g$ with $n$ smooth maked points, a morphism $f : C \rightarrow X$ is a \emph{stable map} of degree $d$ if $f_*([C]) = d$, and,  whenever $f$ contracts a rational component $C_0 \cong \mathbb{P}^1$, $C_0$ is incident to at least $3$ `special' points (marked points or nodes). We denote by $\mbar_{g, n}(X, \beta)$ the moduli stack of stable maps to $X$ of degree $d$. We denote by $\mbar_{g, n}^\bullet(X, d)$ the stack of possibly disconnected stable maps to $X$. The space of stable maps (both connected and possibly disconnected) has virtual dimension

\begin{equation*}
\text{vdim}\left(\mbar^\bullet_{g, n}(X, d)\right) = \int_d c_1(T_X) + (\text{dim}(X) - 3)(1 - g) + n
\end{equation*}

\noindent We can compute intersection numbers on $\mbar_{g, n}^\bullet(X, d)$ by integrating cohomology classes against the virtual fundamental class $\left[\mbar_{g, n}^\bullet(E, d)\right]^{\text{vir}} \in H_{e}(\mbar_{g, n}^\bullet(X, d))$, where $e$ is the virtual dimension.

This paper is concerned with the special case when $X = E$ is a smooth elliptic curve. In this case, the homology class $d$ is assumed to be a nonnegative integer. The space $\mbar_{g, n}^{\bullet}(E, d)$ comes equipped with $n$ evaluation maps $\text{ev}_i: \mbar_{g, n}^{\bullet}(E, d) \rightarrow E, \ 1 \leq i \leq n$. The map $\text{ev}_i$ sends a stable map $\left[f:(C, p_1, \ldots, p_n) \rightarrow E\right]$ to $f(p_i) \in E$. Furthermore, we consider the cohomology classes $\psi_i := c_1(\mathbb{L}_i) \in H^2(\mbar_{g, n}^{\bullet}(E, d))$, where $\mathbb{L}_i$ is the line bundle whose fiber over a point $[f:(C, p_1, \ldots, p_n) \rightarrow E]$ is the cotangent space to the $i^{th}$ marked point $p_i \in C$. These $\psi$-classes are called {\bf{descendent classes}}. We are concerned with intersection numbers of the form

\begin{equation*}
\left<\tau_{k_1}\ldots\tau_{k_n}\right>_d^{\bullet E} := \int_{\left[\mbar_{g, n}^\bullet(E, d)\right]^{\text{vir}}}\psi_1^{k_1}\text{ev}_1^*\left(\left[pt.\right]\right)\ldots\psi_n^{k_n}\text{ev}_n^*\left(\left[pt.\right]\right)
\end{equation*}

These intersection numbers are referred to as {\bf{stationary descendent Gromov-Witten invariants}} of $E$. The word `stationary' is used because the classes $\text{ev}_i^*\left(\left[pt.\right]\right)$ restrain the stable map to send the marked points to a fixed set of points. Our first goal in this section is to answer: \emph{how do we compute stationary descendent GW invariants of E?} Fortunately, there is a nice closed formula for these intersection numbers.

In order to write down this closed formula, we need to recall some basic facts from the representation theory of the symmetric group (see, for example, \cite{sagan2013symmetric}). Let $S_d$ be the symmetric group on $d$ letters. The {\bf{irreducible characters}} of $S_d$ are in bijection with the set of partitions $\left\{\lambda \vdash d\right\}$. We denote the irreducible character corresponding to $\lambda \vdash d$ by $\chi^{\lambda}$. The conjugacy classes of $S_d$ are also indexed by partitions $\mu \vdash d$. This is due to the fact that $\pi_1, \pi_2 \in S_d$ are conjugate to each other if and only if they have the same cycle type $\mu$. Let $\pi_\lambda$ be any permutation of cycle type $\lambda$. The order of the {\bf{centralizer}} of $\pi_\lambda$, that is, the order of the subgroup of elements in $S_d$ that commute with $\pi_\lambda$ only depends on $\lambda$, and is computed by $z_\lambda = \left(\prod_{i = 1}^dm_i!\right)\left(\prod_{i = 1}^{\ell(\lambda)}\lambda_i\right)$, where $m_i$ is the number of times the number $i$ appears in the partition $\lambda$. Since the irreducible character $\chi^{\lambda}$ is constant on conjugacy classes of $S_d$, we can unambiguously define the integer $\chi^{\lambda}_\mu$ as the value of $\chi^{\lambda}$ evaluated at any element in the conjugacy class corresponding to $\mu$. There are two standard ways to compute the numbers $\chi^{\lambda}_{\mu}$, either via \emph{Frobenius' formula} (\cite{fulton2013representation}, Frobenius Formula 4.10), or the \emph{Murnaghan-Nakayama rule} (\cite{fulton2013representation} Problem 4.45).

Next, we need the notion of a {\bf{shifted symmetric power sum}}. Shifted symmetric power sums are examples of \emph{shifted symmetric functions}. The ring $\Lambda^*$ of shifted symmetric functions is a deformation of the standard ring of symmetric functions $\Lambda$. They were introduced and studied by Okounkov and Olshanksi \cite{okounkov1997shifted}, and form a crucial ingredient in the celebrated Bloch-Okounkov Theorem \cite{bloch2000character}. For the purposes of this paper, we only need the following definition.

\begin{Definition}
Let $k \geq 0$ be an integer, and let $\lambda = (\lambda_1, \ldots, \lambda_m) \vdash d$ be a partition. The {\bf{shifted symmetric power sum of order $k$}} evaluated at $\lambda$ is

\begin{equation*}
p_k(\lambda) := \sum_{i = 1}^m \left[ \left( \lambda_i - i + \frac{1}{2} \right)^k - \left( -i + \frac{1}{2} \right)^k \right] - (1 - 2^{-k})\frac{B_{k}}{2k}
\end{equation*}
\end{Definition}

\noindent We now have all of the ingredients to compute the intersection number $\left<\tau_{k_1}\ldots\tau_{k_n}\right>^{\bullet E}_d$.

\begin{Theorem}[Gromov-Witten/Hurwitz Correspondence, \cite{okounkov2006gromov}]\label{GWH}

The intersection numbers $\left<\tau_{k_1}\ldots\tau_{k_n}\right>^{\bullet E}_d$ have the following closed form expression:

\begin{equation*}
\left<\tau_{k_1}\ldots\tau_{k_n}\right>^{\bullet E}_d = \frac{1}{\prod_{i = 1}^n(k_i + 1)!}\sum_{\lambda, \mu \vdash d} \frac{1}{z_\lambda}\left(.\chi^{\lambda}_\mu \right)^2 \prod_{i = 1}^np_{k_i + 1}(\lambda)
\end{equation*}

\end{Theorem}

A formal power series that will come up often in our calculations is the so-called {\bf{Euler function}}, defined by

\begin{equation}\label{euler}
(q)_\infty := \prod_{n \geq 1}(1 - q^n) = \sum_{d = -\infty}^\infty (-1)^dq^{\frac{3d^2 - d}{2}}
\end{equation}

\noindent The second equality in Equation \ref{euler} is often referred to as the \emph{Pentagonal Number Theorem}. It will help us write down closed-form expressions for Ramanujan tau values in terms of Gromov-Witten invariants. The individual intersection numbers $\left<\tau_{k_1}\ldots\tau_{k_n}\right>^{\bullet E}_d$ are not our main focus, but rather, we want to stitch them together into the following generating functions:

\begin{Definition}
Let $k_1, \ldots, k_n$ be nonnegative integers. We call
\begin{equation*}
\left<\tau_{k_1}\ldots\tau_{k_n}\right> := (q)_\infty \sum_{d \geq 0}\left<\tau_{k_1}\ldots\tau_{k_n}\right>^{\bullet E}_d q^d
\end{equation*}
\noindent a {\bf{stationary descendent of weight $k = \sum(k_i + 2)$}}
\end{Definition}

\begin{Theorem}\label{opquasimodular}[\cite{okounkov2006gromov}]
Stationary descendents of weight $k$ are quasimodular forms of weight $k$, that is, if $\sum_{i = 1}^nk_i = k$, $\left<\tau_{k_1}\ldots\tau_{k_n}\right> \in \text{QM}_{k}$ 
\end{Theorem}

\begin{Example}
It's a straightforward calculation to show that
\begin{equation*}
\left<\tau_0\right> = E_2
\end{equation*}

\noindent Here is a less trivial example in weight $8$:

\begin{equation*}
\left<\tau_2^2\right> = \frac{1}{12}E_6E_2 + \frac{73}{112}E_4^2 - \frac{3}{4}E_4E_2^2 - \frac{15}{4}E_2^3
\end{equation*}

\noindent In the Appendix, we explain how one can use the author's Sage code to make such calculations.
\end{Example}

Theorem \ref{opquasimodular} allows us to go from GW invariants to quasimodular forms. But it turns out that we can also go the other way around. 

\begin{Proposition}\label{qmviadescendents}
Let $(\left<\tau_{\text{wt.} = 2}\right>, \left<\tau_{\text{wt.} = 4}\right>, \left<\tau_{\text{wt.} = 6}\right>)$ be a triple of stationary descendents of weight $2, 4$, and $6$, respectively. Then

\begin{equation*}
\text{QM} = \mathbb{Q}\left[\left<\tau_{\text{wt.} = 2}\right>,  \left<\tau_{\text{wt.} = 4}\right>, \left<\tau_{\text{wt.} = 6}\right> \right]
\end{equation*}

\end{Proposition}

\begin{proof}

Let $(\left<\tau_{\text{wt.} = 2}\right>, \left<\tau_{\text{wt.} = 4}\right>, \left<\tau_{\text{wt.} = 6}\right>)$ be a triple of stationary descendents in weights $2, 4$, and $6$, respectively. The descendent $\left<\tau_{\text{wt.}=2}\right>$ is necessarily $\left<\tau_0\right> = E_2$, so \\ $E_2 \in \mathbb{Q}\left[\left<\tau_{\text{wt.} = 2}\right>,  \left<\tau_{\text{wt.} = 4}\right>, \left<\tau_{\text{wt.} = 6}\right> \right]$. There are two choices for $\left<\tau_{\text{wt.} = 4}\right>$ and $4$ choices for $\left<\tau_{\text{wt.}=6}\right>$ (see Figure \ref{eighttypes}).  Since $\left<\tau_{\text{wt.}=4}\right> \in \text{QM}_4$ and $\left<\tau_{\text{wt.}=6}\right> \in \text{QM}_6$, there exists rational numbers $a_i, 1 \leq i \leq 5$, such that

\begin{align*}
\left<\tau_{\text{wt.}=4}\right> & = a_1E_4 + a_2E_2^2 \\
\left<\tau_{\text{wt.=6}}\right> & = a_3E_6 + a_4E_4E_2 + a_5E_2^3
\end{align*} 

\noindent In order to show that $E_4 \in \mathbb{Q}\left[\left<\tau_{\text{wt.} = 2}\right>,  \left<\tau_{\text{wt.} = 4}\right>, \left<\tau_{\text{wt.} = 6}\right> \right]$, we need to show that $a_1 \not= 0$. Indeed, for the two choices of $\left<\tau_{\text{wt.}=4}\right>$, we have

\begin{align}\label{weightfour}
\left<\tau_2\right> & = \frac{1}{12}E_4 + \frac{1}{2}E_2^2 \\
\notag \left<\tau_0^2\right> & = \frac{5}{6}E_4 - E_2^2
\end{align}

\noindent and it follows that $E_4 \in \mathbb{Q}\left[\left<\tau_{\text{wt.} = 2}\right>,  \left<\tau_{\text{wt.} = 4}\right>, \left<\tau_{\text{wt.} = 6}\right> \right]$. Similarly, since

\begin{align}\label{weightsix}
\left<\tau_4\right> & = \frac{1}{360}E_6 + \frac{1}{12}E_4E_2 + \frac{1}{6}E_2^3 \\
\notag \left<\tau_2\tau_0\right> & = \frac{7}{120}E_6 + \frac{1}{4}E_4E_2 - \frac{3}{2}E_2^3  \\
\notag \left<\tau_1^2\right> & = \frac{7}{180}E_6 + \frac{2}{3}E_4E_2 - \frac{8}{3}E_2^3\\
\notag \left<\tau_0^3\right> & = \frac{7}{12}E_6 - \frac{15}{2}E_4E_2 + 3E_2^3
\end{align}

\noindent it follows that $E_6 \in \mathbb{Q}\left[\left<\tau_{\text{wt.} = 2}\right>,  \left<\tau_{\text{wt.} = 4}\right>, \left<\tau_{\text{wt.} = 6}\right> \right]$.

\end{proof}

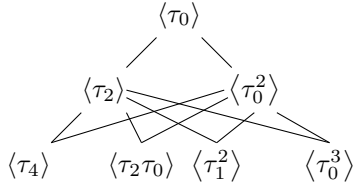
\begin{figure}
\centering
\begin{tikzpicture}

\draw (0, 0) node {$\left<\tau_0\right>$};
\draw (-1, -1) node {$\left<\tau_2\right>$};
\draw (1, -1) node {$\left<\tau_0^2\right>$};
\draw (-2, -2) node {$\left<\tau_4\right>$};
\draw (-0.5, -2) node {$\left<\tau_2\tau_0\right>$};
\draw (0.5, -2) node {$\left<\tau_1^2\right>$};
\draw (2, -2) node {$\left<\tau_0^3\right>$};

\draw (0.3, -0.3) -- (0.7, -0.7);
\draw (-0.3, -0.3) -- (-0.7, -0.7);

\draw (-1.3, -1.3) -- (-1.7, -1.7);
\draw (-0.7, -1.3) -- (-0.5, -1.7);
\draw (-0.7, -1.1) -- (0.5, -1.7);
\draw (-0.7, -1) -- (2, -1.7);

\draw (0.6, -1) -- (-1.7, -1.7);
\draw (0.65, -1.1) -- (-0.5, -1.7);
\draw (1, -1.3) -- (0.5, -1.7);
\draw (1.3, -1.3) -- (2, -1.7);

\end{tikzpicture}
\caption{A simple visualization to construct all $8$ triples of stationary descendents that generate the ring of quasimodular forms. The top row corresponds to the weight $2$ stationary descendent, the middle row the weight $4$ descendents, and the bottom row the weight $6$ descendents. Each triple corresponds to a path from the top row to the bottom.}
\label{eighttypes}
\label{visualizingtypes}
\end{figure}

Let $p(d)$ be the {\bf{partition function}} i.e. $p(d) = \#\{\text{partitions of the integer $d$}\}$. The triple $\left(\left<\tau_0\right>, \left<\tau_0^2\right>, \left<\tau_0^3\right>\right)$ is particularly nice since there is a closed-form expression for $\left<\tau_0^n\right>$ in terms of the partition function $p(d)$. 

\begin{Proposition}\label{allzeroes}

We have

\begin{equation*}
24^n\left<\tau_0^n\right> = (q)_\infty\sum_{d \geq 0}(24d - 1)^np(d)q^d
\end{equation*}

\end{Proposition}

\begin{proof}

It is a standard fact that, for all partitions $\lambda \vdash d$,

\begin{equation*}
\sum_{\mu \vdash d}\left( \chi^\lambda_\mu\right)^2 = z_\lambda
\end{equation*}

\noindent Therefore, we have

\begin{equation}\label{partition}
\sum_{\lambda, \mu \vdash d}\frac{1}{z_\lambda}\left(\chi^\lambda_\mu\right)^2 = \sum_{\lambda \vdash d} \frac{1}{z_\lambda}\sum_{\mu \vdash d} \left( \chi^\lambda_\mu\right)^2 = \sum_{\lambda \vdash d} 1 = p(d)
\end{equation}

\noindent Furthermore, since $p_1(\lambda) = d - \frac{1}{24}$, the desired result then follows from the following computation,

\begin{align*}
24^n\left<\tau_0^n\right> & = 24^n(q)_\infty\sum_{d \geq 0}\left<\tau_0^n\right>^{\bullet E}_dq^d \hspace{4.7cm} \left(\text{By definition}\right) \\
& = 24^n(q)_{\infty}\sum_{d \geq 0} \left( \sum_{\lambda, \mu \vdash d}\frac{1}{z_\lambda}\left(\chi^\lambda_\mu\right)^2 \left(d - \frac{1}{24}\right)^n \right)q^d \hspace{1cm} \left(\text{Since $p_1(\lambda) = d - \frac{1}{24}$}\right) \\
& = (q)_\infty\sum_{d \geq 0}(24d - 1)^np(d)q^d \hspace{4.4cm} (\text{By Equation \ref{partition}})
\end{align*}

\end{proof}

Proposition \ref{qmviadescendents} tells us that, in addition to the triple $(E_2, E_4, E_6)$, we can use any triple $(\left<\tau_{\text{wt.} = 2}\right>, \left<\tau_{\text{wt.} = 4}\right>, \left<\tau_{\text{wt.} = 6}\right>)$ of stationary descendents of weight $2, 4$, and $6$, respectively, to generate $\text{QM}$. Furthermore, there are only $8$ such triples (see Figure \ref{eighttypes}). In order to distinguish whether we are generating $\text{QM}_k$ in terms of Eisenstein series or stationary descendents, we make the following definition.

\begin{Definition}
Let $(\left<\tau_{\text{wt.} = 2}\right>, \left<\tau_{\text{wt.} = 4}\right>, \left<\tau_{\text{wt.} = 6}\right>)$ be a triple of of stationary descendents in weights $2, 4$, and $6$, respectively. For any $k \in 2\mathbb{Z}$, the collection of forms 

\begin{equation*}
\left\{ \left<\tau_{\text{wt,} = 2}\right>^a\left<\tau_{\text{wt.}=4}\right>^b\left<\tau_{\text{wt.}=6}\right>^c  \right\}_{2a + 4b + 6c = k}
\end{equation*}

\noindent is the {\bf{descendent basis of weight $k$}} corresponding to the triple $(\left<\tau_{\text{wt.} = 2}\right>, \left<\tau_{\text{wt.} = 4}\right>, \left<\tau_{\text{wt.} = 6}\right>)$. The collection

\begin{equation*}
\left\{ E_2^aE_4^bE_6^c  \right\}_{2a + 4b + 6c = k}
\end{equation*}

\noindent is the {\bf{Eisenstein basis of weight $k$}}.
\end{Definition}

\begin{Corollary}\label{nonlineardiscriminant}
The discriminant modular form $\Delta$ has $8$ different expressions in terms of the $8$ descendent bases of $\text{QM}_{12}$.
\end{Corollary}

\noindent All $8$ expressions of $\Delta$ mentioned in Corollary \ref{nonlineardiscriminant} are collected in Table \ref{descendentbases}. An unexpected result of this calculation is that, the discriminant modular form always lies in the smaller subring $\mathbb{Z}\left[\left<\tau_{\text{wt.=2}}\right>, \left<\tau_{\text{wt.}=4}\right>, \left<\tau_{\text{wt,}=6}\right> \right] \subset \mathbb{Q}\left[ \left<\tau_{\text{wt.=2}}\right>, \left<\tau_{\text{wt.}=4}\right>, \left<\tau_{\text{wt,}=6}\right> \right]$. It would be interesting to know if there is an intrinsic explanation for this, and whether this occurs for other modular forms as well.

\section{The Descendent Matroid of Weight $k$}\label{descendentmatroid}

Corollary \ref{nonlineardiscriminant} tells us that we have $8$ different ways of expressing $\Delta$ in the descendent bases of weight $12$. We want to be able to use these expressions to find a closed-form expression of the Ramanujan tau function $\tau(d)$ i.e. we want to extract the degree-$d$ coefficient in each of the entries of Table \ref{descendentbases} in a closed form. However, when one attempts such a calculation, a difficulty arises. To see how this difficulty manifests itself, consider the first entry of Table \ref{descendentbases},

\begin{align*} 
\Delta & = 13824\left<\tau_0^2\right>^3 + 6480\left<\tau_0^2\right>^2\left<\tau_0\right>^2 \\ 
& - \left( 432\left<\tau_0^3\right>^2 + 7776\left<\tau_0^3\right>\left<\tau_0^2\right>\left<\tau_0\right> + 5184\left<\tau_0^3\right>\left<\tau_0\right>^3  + 5184\left<\tau_0^2\right>\left<\tau_0\right>^4 + 1728\left<\tau_0\right>^6\right) 
\end{align*}

\noindent Already in the first term of the above expression, we are confronted with the power series $(q)_\infty^3 = \prod(1 - q^n)^3$. In general, there is no closed form expression for the $d^{th}$ coefficient of $(q)_\infty^n$ i.e. there is no analogue of the Pentagonal Number Theorem for higher powers of $(q)_\infty$. 

Of course, the remedy to this problem is clear: if we can find $\mathbb{Q}$-bases of $\text{QM}_{12}$ whose elements are stationary descendents of weight $12$ (as opposed to \emph{monomials} of stationary descendents), then we can express $\Delta$ in these bases,  and apply the Pentagonal Number Theorem to extract $\tau(d)$. 

Before we employ this strategy in weight $12$, let's see if we can answer the simpler yet related question: \emph{Is it possible to find such bases for $\text{QM}_k$ for weights smaller than 12?} It is precisely our ability to answer this question in the affirmative that will naturally lead us to structure our thinking using the language of matroids. 

To motivate the introduction of matroids, let's first examine the simple cases of weights $k = 2, 4, 6$. The situation is trivial when $k = 2$, since $\text{QM}_2 = \text{span}_\mathbb{Q}\left( \left<\tau_0\right> \right)$. In $\text{QM}_4$, the only stationary descendents of weight $4$ are $\left<\tau_2\right>$ and $\left<\tau_0^2\right>$. We claim that these two descendents comprise a basis of $\text{QM}_4$. Indeed, in the proof of Proposition \ref{qmviadescendents}, we expanded these stationary descendents in the Eisenstein basis in Equation \ref{weightfour}, and obtained $\left<\tau_2\right> = \frac{1}{12}E_4 + \frac{1}{2}E_2^2$, and $\left<\tau_0^2\right> = \frac{5}{6}E_4 - E_2^2$. Since the matrix 

\begin{equation*}
A_4 := \left[
\begin{matrix} 
\frac{1}{12} & \frac{5}{6} \\[6pt] 
\frac{1}{2} & -1 
\end{matrix}
\right]
\end{equation*}

\noindent is full rank, this proves the claim. Similarly, for $\text{QM}_6$, there are four stationary descendents of weight six: $\left<\tau_4\right>, \left<\tau_2\tau_0\right>, \left<\tau_1^2\right>, \left<\tau_0^3\right>$. We expanded these descendents in terms of Eisenstein series in Equation \ref{weightsix}. Consider the corresponding matrix,

\begin{equation*}
A_6 := \left[
\begin{matrix}
\frac{1}{360} & \frac{7}{120} & \frac{7}{180} & \frac{7}{12} \\[6pt]
\frac{1}{12}   & \frac{1}{4}     & \frac{2}{3}     & -\frac{15}{2} \\[6pt]
\frac{1}{6}     & -\frac{3}{2}    & -\frac{8}{3}    & 3
\end{matrix}
\right]
\end{equation*}

\noindent Not only is the  above matrix of full rank, but it has the additional property that, any collection of three of its columns forms a basis of $\mathbb{Q}^3 \cong \text{QM}_6$. In the language of matroids, this phenomenon is the occurrence of a {\bf{uniform matroid}}. 

Let us now introduce matroids properly. As was said in the Introduction, matroids are combinatorial objects that abstract the notion of linear independence/dependence.

\begin{Definition}[\cite{oxley2011matroid}, Section 1.1]
A {\bf{matroid}} $\mathcal{M}$ is an ordered pair $(E, \mathcal{I})$, where $E$ is a finite set, and $\mathcal{I}$ is a collection of subsets of $E$ satisfying the following properties:
\begin{enumerate}
\item{$\varnothing \in \mathcal{I}$}
\item{$I \in \mathcal{I}, I' \subseteq I \implies I' \in \mathcal{I}$}
\item{$I_1, I_2 \in \mathcal{I}, |I_1| < |I_2| \implies \text{there exists} \ e \in I_2\setminus I_1 \ \text{such that} \ I_1 \cup e \in \mathcal{I}$}
\end{enumerate}
The set $E$ is called the \emph{ground set} of $\mathcal{M}$, and the elements of $\mathcal{I}$ are called the \emph{independent sets} of $\mathcal{M}$. 
\end{Definition}

\noindent A standard way to construct a matroid is to start with an $(m \times n)$-matrix $A$ with entries from a field $\bf{k}$, declare the ground set to be the columns of the matrix, and let the independent sets be the collection of columns that are linearly independent in the vector space $\bf{k}^m$. The matroid constructed this way is denoted $M[A]$, and is called the {\bf{vector matroid}} associated to $A$. The matroid $M[A]$ is completely determined by its maximal independent sets i.e. the {\bf{bases}} of $M[A]$. The {\bf{rank}} of a matroid is the length of any of its maximal independent sets. When describing a matroid by its bases, we use the ordered pair $\mathcal{M} = (E, \mathcal{B})$, where $\mathcal{B}$ is the set of bases of $\mathcal{M}$.  The theory of matroids, along with what we have established in weights $k = 2, 4, 6$, motivates the following definition.   

\begin{Definition}\label{descendentmatroid}
Let $k \in 2\mathbb{Z}$ be a positive even integer. Define the ground set $E^k$ to be
\begin{equation*}
E^k := \left\{ \left<\prod_{i = 1}^n\tau_{k_i} \right> \right\}_{\sum(k_i + 2) = k}
\end{equation*}
\noindent and let $\mathcal{I}^k$ be the subsets of $E^k$ that form a linearly independent set in the vertor space $\text{QM}_k$. We call the ordered pair $\mathcal{M}_k := (E^k, \mathcal{I}^k)$ the {\bf{descendent matroid of weight $k$}}. We denote by $\mathcal{B}^k$ the elements of $\mathcal{I}^k$ of length $\text{dim}(\text{QM}_k)$.
\end{Definition} 

\noindent A priori, it may turn out that $\mathcal{B}^k$ is empty. However, we conjecture that this is not the case (see Conjecture \ref{fullrank}). 

Fix an ordering of the Eisenstein basis of $\text{QM}_k$, say, $v_1, v_2, \ldots, v_{\text{dim}(\text{QM}_k)}$, along with an ordering of the stationary descendents of weight $k$, say, $\left<\tau\right>_1, \ldots, \left<\tau\right>_{\left|E^k\right|}$. Define $( \ , \ ): \text{QM}_k \times \text{QM}_k \rightarrow \mathbb{Q}$ to be the inner product on $\text{QM}_k$ given by $(v_i, v_j) = \delta_{ij}$, and let $A_k = (a_{ij})$ be the $\text{dim}(\text{QM}_k) \times \left| E^k \right|$ matrix where $a_{ij} = \left(v_i, \left<\tau\right>_j\right)$. Then $\mathcal{M}_k \cong M\left[A_k\right]$, and this is true independent of the ordering we have chosen for the Eisenstein basis or the stationary descendents. 

The {\bf{uniform matroid}} $U_{r, n}$ is the rank-$r$ matroid with ground set $[n] := \{1, \ldots, n\}$ such that all $r$-element subsets of $[n]$ are bases. From our previous calculations, we can conclude that:

\begin{Proposition}
For $k = 2, 4, 6$, $\mathcal{M}_k$ is the uniform matroid $U_{\text{dim}(\text{QM}_k), \left| E^k \right|}$.
\end{Proposition}

\noindent When $k = 8$, if we order the stationary descendents as 

\begin{equation*}
\left[\left<\tau_6\right>, \left<\tau_4\tau_0\right>, \left<\tau_3\tau_1\right>, \left<\tau_2\tau_2\right>, \left<\tau_2\tau_0^2\right>, \left<\tau_1^2\tau_0\right> \left<\tau_0^4\right>\right]
\end{equation*}

\noindent and we order the Eisenstein basis of weight $8$ as $\left[E_6E_2, E_4^2, E_4E_2^2, E_2^4\right]$, we have

\begin{equation*}
A_8 = \left[
\begin{matrix}
\frac{1}{360} & \frac{1}{36} & \frac{13}{180} & \frac{1}{12} & -\frac{7}{15} & \frac{7}{180} & -\frac{35}{3} \\[6pt]
\frac{19}{2016} & \frac{115}{504} & \frac{25}{63} & \frac{73}{112} & \frac{85}{24} & \frac{25}{9} & \frac{352}{12} \\[6pt]
\frac{1}{24} & -\frac{1}{3} & -\frac{2}{3} & -\frac{3}{4} & -6 & -\frac{38}{3} & 75 \\[6pt]
\frac{1}{24} & -\frac{5}{6} & -\frac{8}{3} & -\frac{15}{4} & \frac{15}{2} & \frac{40}{3} & -15
\end{matrix}
\right]
\end{equation*}

\noindent Using the matrix representation above, we have the following:

\begin{Proposition}\label{weighteight}
Let $M = (E, \mathcal{B})$ be the rank-$4$ matroid with ground set $E = (abcdefg)$, whose bases $\mathcal{B}$ are given by
\begin{equation*}
\mathcal{B} = \{\text{$4$-elements subsets of E}\}\setminus \{abcd\}
\end{equation*} 
\noindent Then $\mathcal{M}_8 \cong M$.
\end{Proposition}

\begin{proof}
The second, fifth, sixth, and seventh column of the matrix $A_8$ forms a submatrix of rank 3. Furthermore, since the total number of bases of $\mathcal{M}_8$ is $34$ (see Appendix \ref{descendentmatroidclass}), and since ${7 \choose 4} = 35$, these columns are the only set of $4$ columns not forming a basis of $\mathbb{Q}^4$.
\end{proof}

In particular, this shows that $\mathcal{M}_8 = M[A_8]$ is not the uniform matroid $U_{4, 7}$. However, what's peculiar is that the only $4$-element subset of $E^8$ that is not a basis of $\text{QM}_8$ is $\{ \left<\tau_4\tau_0\right>, \left<\tau_2\tau_0^2\right>, \left<\tau_1^2\tau_0\right>, \left<\tau_0^4\right>\}$, which happens to be the $4$-element subset that contains all stationary descendents with an insertion of $\tau_0$. This suggests that the stationary descendents whose insertions are strictly positive may play a distinguished role in minimally and symmetrically generating the quasimodular forms of weight $k$. Indeed, this intuition will be buttressed as we see more examples.

\begin{Definition}
A stationary descendent $\left<\tau_{k_1}\ldots\tau_{k_n}\right>$ of weight $k$ is called a {\bf{positive stationary descendent}} if $k_i > 0$ for all $i$. We denote by $S^k \subset E^k$ the set of all positive stationary descendents of weight $k$. 
\end{Definition}

In order to proceed, we need to introduce a standard operation on matroids, namely, the operation of {\bf{restriction/deletion}}. If $\mathcal{M} = (E, \mathcal{I})$ is a matroid, and $S \subset E$, the {\bf{restricted matroid}} $\mathcal{M} | S$ is the matroid whose ground set is $S$ and whose independent sets are given by $\mathcal{I} | S := \left \{ I \in \mathcal{I} : I \subset S \right \}$. Taking restrictions of matroids is a useful operation in matroid theory because it allows one to understand a matroid $\mathcal{M}$ by its various substructures i.e. its {\bf{minors}}. 

Now consider the matroid $\mathcal{M}_{10}$. This matroid has much more combinatorial complexity. For example, Sage can tell us that it has $730$ bases. However, it turns out that $\mathcal{M}_{10} | S^{10} = U_{5, 5}$, giving further credence that $\mathcal{M}_k | S^k$ is a nice restriction.

Now we can finally examine the weight that we were originally interested in, $k = 12$. A matrix representation would require a $(7 \times 21)$ matrix, which we don't include in this document, but we show the reader how to compute this matrix using our Sage code (see Appendix \ref{descendentmatroidclass}). The combinatorial complexity of $\mathcal{M}_{12}$ is immense: according to Sage, $\mathcal{M}_{12}$ has $102670$ bases. However, when we consider the restricted matroid $\mathcal{M}_{12} | S^{12}$, it has a more manageable $(7 \times 9)$ matrix representation. Furthermore, this restricted matroid has the property that $\left| \mathcal{B}^{12} | S^{12} \right| = {9 \choose 7} = 36$, thereby proving that:

\begin{Theorem}
The matroid $\mathcal{M}_{12} | S^{12}$ is is the uniform matroid $U_{7, 9}$ 
\end{Theorem}

Let's take a step back and see what has been done so far. First off, notice that our matroidal escapades have led us to a naturally defined and combinatorially symmetric way of generating $\text{QM}_{12}$ with stationary descendents of weight $12$. Secondly, we have solved the issue that we mentioned at the beginning of this chapter, namely, we have found not one, but $36$ bases of $\text{QM}_{12}$ whose elements are not monomials of stationary descendents, but simply just stationary descendents. We can now express the discriminant modular form $\Delta \in \text{QM}_{12}$ as a linear combination of stationary descendents of weight 12 in 36 different ways. These are collected in Tables \ref{partone}, \ref{parttwo}, and \ref{partthree}. Translating this to a statement about the Ramanujan tau function, we have the following corollary:   

\begin{Corollary}\label{fundamentalcorollary}
Let $\tau(d)$ be the Ramanujan tau function evaluated at the integer $d$. For any basis $B \in \mathcal{B}^{12} | S^{12}$, and for any stationary descendent $b \in B$, denote by $\left<\tau_b\right>^{\bullet E}_d$ the degree-$d$ coefficient of $b/(q)_{\infty}$. Then there exists a unique vector $(a_b)_{b \in B} \in \mathbb{Q}^7$ such that
\begin{equation}\label{cancellation}
\tau(d) = \sum_{b \in B \atop 3d_1^2 - d_1 + 2d_2 = 2d} (-1)^{d_1}a_b\left<\tau_b\right>^{\bullet E}_{d_2} 
\end{equation}
\end{Corollary}

\begin{proof}
Let $B \in \mathcal{B}^{12} | S^{12}$ be a basis. Then there exists a unique vector $(a_b)_{b \in B} \in \mathbb{Q}^7$ such that

\begin{equation}\label{corollaryequality}
\Delta = \sum_{b \in B}a_bb = (q)_\infty \sum_{b \in B} a_b \left(\sum_{d \geq 0}\left<\tau_b\right>^{\bullet E}_d q^d\right)
\end{equation}

\noindent The Pentagonal Number Theorem says that

\begin{equation*}
(q)_\infty = \sum_{d = -\infty}^\infty (-1)^dq^{\frac{3d^2 - d}{2}}
\end{equation*}

\noindent The desired result follows from matching up the degree-$d$ coefficients of both side of Equation \ref{corollaryequality}.
\end{proof}

Corollary \ref{fundamentalcorollary} is not particularly useful for computing values of $\tau(d)$ for large $d$. However, it does provide qualitative evidence of Lehmer's conjecture: if Lehmer's conjecture is false, this would have to result from a large confluence of many terms cancelling on the right-hand-side of Equation \ref{cancellation}. 

A nice consequence of Corollary \ref{fundamentalcorollary} is the ability to translate every known property of $\tau(d)$ into a relation of Gromov-Witten invariants. For example, since $\tau(d)$ is a multiplicative function i.e. $\tau(m)\tau(n) = \tau(mn)$ for $(m, n) = 1$, Corollary \ref{fundamentalcorollary} becomes

\begin{align*}
\left(\sum_{b \in B \atop 3d_1^2 - d_1 + 2d_2 = 2m} (-1)^{d_1}a_b\left<\tau_b\right>^{\bullet E}_{d_2}\right) &  \left(\sum_{b \in B \atop 3d_1^2 - d_1 + 2d_2 = 2n} (-1)^{d_1}a_b\left<\tau_b\right>^{\bullet E}_{d_2}    \right) \\
& = \sum_{b \in B \atop 3d_1^2 - d_1 + 2d_2 = 2mn} (-1)^{d_1}a_b\left<\tau_b\right>^{\bullet E}_{d_2} 
\end{align*}

\noindent which, as far as we know, is a new relation of Gromov-Witten invariants for the elliptic curve. Similar relations can be derived by choosing different properties of $\tau(d)$ and seeing what it says on the Gromov-Witten side.

\subsection{Outlook}

It would be worthwhile to further develop the theory of descendent matroids, and to see which modular forms have nice representations in terms of stationary descendents. As was mentioned after Definition \ref{descendentmatroid}, it's not clear whether $\mathcal{B}^k$ is always nonempty, but we conjecture that this is indeed the case.

\begin{Conjecture}\label{fullrank}
The rank of the descendent matroid $\mathcal{M}_k$ is always equal to $\text{dim}(\text{QM}_k)$. Equivalently,
\begin{equation*}
\mathbb{Q}\left[ \left\{ \left< \tau_{k_1}\ldots\tau_{k_n} \right> \right\}_{\sum(k_i + 2) = k} \right] = \text{QM}_k
\end{equation*}
\end{Conjecture}

To check that Conjecture \ref{fullrank} is at least plausible, we should make sure the ground set $E^k$ is at least as large as $\text{dim}(\text{QM}_k)$. By the following lemma, this is indeed true:

\begin{Lemma}
The number of elements in the ground set $E^k$ is
\begin{equation*}
\left| E^k \right| = \# \{\text{partitions of $k$ with no part equal to $1$} \}
\end{equation*}
\end{Lemma}

\begin{proof}
Let $\mathcal{P}_k^{> 1}$ be the set of partitions of $k$ with no part equal to 1. The map $\varphi : \mathcal{P}_k^{> 1} \rightarrow E^k$ given by $(\lambda_1, \ldots, \lambda_n) \mapsto \left<\tau_{\lambda_1 - 2}\ldots\tau_{\lambda_n - 2}\right>$ is a bijection.
\end{proof}

\noindent In particular, $\left|E^k\right| \gg \text{dim}(\text{QM}_k) = \#\{\text{partitions of $\frac{k}{2}$ with no parts greater than 3}\}$ for large $k$.

The calculations we have made in this paper suggest that the ground set $E^k$ contains many redundant generators of $\text{QM}_k$, but that a good portion of this redundancy can be killed by restricting to $S^k$. However, in order to find minors that are nontrivial uniform matroids, further restriction seems to be the key. For example,

\begin{align*}
\mathcal{M}_{14} \ | \ \left(S^{14} \setminus \{\text{$4$-pointed descendents}\} \cup\left<\tau_3^2\tau_2\right> \right) & = U_{8, 10} \\
\mathcal{M}_{16} \ | \ \left(S^{16} \setminus \{\text{$4$ and $5$-pointed descendents}\} \cup \left<\tau_4\tau_3^2\right> \right) & = U_{10, 14} \\
\mathcal{M}_{18} \ | \ \left( S^{18} \setminus \{\text{$4$, $5$, and $6$-pointed descendents}\} \cup \{\left<\tau_4^3\right>, \left<\tau_5\tau_4\tau_3\right>, \left<\tau_5^2\tau_2\right>, \left<\tau_6\tau_3^2\right>\} \right) & = U_{12, 16}
\end{align*}

\begin{Conjecture}\label{minor}
For $k \geq 14$, there always exists an $n > \text{dim}(\text{QM}_k)$ such that the matroid $\mathcal{M}_k | S^k$ has $U_{\text{dim}(\text{QM}_k), n}$ as a minor.
\end{Conjecture}

\section{Appendix}

The calculations used in this paper were implemented using Sage \cite{sagemath}. At the author's website \cite{Website}, there is a text file, descendent.txt, containing the classes and functions used for the calculations. We have also provided a Jupyter notebook (also found at the author's website) that serves as a tutorial on how to use the code. The two main classes in descendent.txt are Stationary\_Descendent(\ ) and Descendent\_Matroid(\ ). 

\subsection{The Stationary\_Descendent(\ ) Class}

When instantiating Stationary\_Descendent(\ ), one specifies the descendent insertions via a list $[k_1, k_2, \ldots, k_n]$. Several methods are available for this object. The .evaluate() method returns the Gromov-Witten invariant $\left<\tau_{k_1}\ldots\tau_{k_n}\right>^{\bullet E}_d$. For example, suppose we wanted to compute $\left<\tau_2^2\right>^{\bullet E}_3$. We would run the following commands

\begin{verbatim}
sage: desc = Stationary_Descendent([2, 2])
sage: desc.evaluate(d=3)
sage: 166577809/11059200
\end{verbatim}

\noindent which shows that $\left<\tau_2^2\right>^{\bullet E}_3 = \frac{166577809}{11059200}$. The method .expand() returns the expansion of the stationary descendent $\left<\tau_{k_1}\ldots\tau_{k_n}\right>$ up to a prescribed order. For example, we have

\begin{verbatim}
sage: desc.expand(order = 3)
sage: 248437/17280*q^3 + 15703/23040*q^2 + 127/69120*q + 49/33177600
\end{verbatim} 

\noindent which means that

\begin{equation*}
\left<\tau_2^2\right> := (q)_\infty\sum_{d \geq 0} \left<\tau_2^2\right>^{\bullet E}_d = \frac{49}{33177600} + \frac{127}{69120}q + \frac{15703}{23040}q^2 + \frac{248437}{17280}q^3 + O(q^4)
\end{equation*}

\noindent The .to\_eisenstein\_basis(\ ) method returns the expansion of $\left<\tau_{k_1}\ldots\tau_{k_n}\right>$ in the Eisenstein basis. Specifically, the output is a dictionary whose keys are tuples corresponding to Eisenstein monomials, and whose values are the corresponding coefficients. For example,

\begin{verbatim}
sage: desc.to_eisenstein_basis()
sage: {(6, 2): 1/12, (4, 4): 73/112, (4, 2, 2): -3/4, (2, 2, 2, 2): -15/4}
\end{verbatim}

\noindent means that

\begin{equation*}
\left<\tau_2^2\right> = \frac{1}{12}E_6E_2 + \frac{73}{112}E_4^2  - \frac{3}{4}E_4E_2^2 - \frac{15}{4}E_2^4
\end{equation*}

\subsection{The Descendent\_Matroid(\ ) Class}\label{descendentmatroidclass}

When instantiating the Descendent\_Matroid(\ ) class, the user specifies a weight $k$. One then has access to various methods to make computations with the descendent matroid $\mathcal{M}_k$. For example, we can compute the matrix representation of $\mathcal{M}_k$ with the .matrix\_repr(\ ) method.

\begin{verbatim}
sage: m4 = Descendent_Matroid(4)
sage: m6 = Descendent_Matroid(6)
sage: m8 = Descendent_Matroid(8)
sage: for matroid in [m4, m6, m8]:
          	print(matroid.matrix_repr())
	 	         print("")
\end{verbatim}

\begin{verbatim}
sage: [1/12  5/6]
      [ 1/2   -1]

[1/360 7/120 7/180  7/12]
[ 1/12   1/4   2/3 -15/2]
[  1/6  -3/2  -8/3     3]

[  1/360    1/36  13/180    1/12   -7/15   7/180   -35/3]
[19/2016 115/504   25/63  73/112   85/24    25/9  325/12]
[   1/24    -1/3    -2/3    -3/4      -6   -38/3      75]
[   1/24    -5/6    -8/3   -15/4    15/2    40/3     -15]
\end{verbatim}

\noindent In order to determine which column corresponds to which stationary descendent, we can use the .groundset(\ ) method. For example, for $\mathcal{M}_8$, we have

\begin{verbatim}
sage: m8.groundset()
sage: [[6], [4, 0], [3, 1], [2, 2], [2, 0, 0], [1, 1, 0], [0, 0, 0, 0]]
\end{verbatim} 

\noindent which means that the first column corresponds to $\left<\tau_6\right>$, the second column corresponds to $\left<\tau_4\tau_0\right>$, and so on.

There is a Sage library for working with matroids. We have baked in this functionality into the Descendent\_Matroid(\ ) class via the method .matroid\_repr(\ ). For example,

\begin{verbatim}
sage: mat8 = m8.matroid_repr(); mat8
sage: Linear matroid of rank 4 on 7 elements represented over the Rational Field
\end{verbatim}

\noindent We now have access to all of the matroid functionality of the Matroids Sage library. For example,

\begin{verbatim}
sage: mat8.rank()
sage: 4
sage: mat8.tutte_polynomial()
sage: x^4 + 3*x^3 + y^3 + 6*x^2 + x*y + 4*y^2 + 9*x + 9*y
sage: mat8.bases_count()
sage: 34
\end{verbatim}

\noindent Notice that the last output above provides one of the calculations needed for the proof of  Proposition \ref{weighteight}. 

If we want see what the bases are for any given matroid $\mathcal{M}_k$, we can use the .bases(\ ) method. It returns an iterator object that contains each basis in the matroid. For example, in $\mathcal{M}_6$, 

\begin{verbatim}
sage: for basis in m6.bases():
          	print(i)
sage: [[4], [2, 0], [1, 1]]
      [[4], [1, 1], [0, 0, 0]]
      [[2, 0], [1, 1], [0, 0, 0]]
      [[4], [2, 0], [0, 0, 0]]
\end{verbatim}

\noindent This means that the four bases of $\mathcal{M}_6$ are $(\left<\tau_4\right>, \left<\tau_2\tau_0\right>, \left<\tau_1^2\right>), (\left<\tau_4\right>, \left<\tau_1^2\right>, \left<\tau_0^3\right>), $ \\
$ (\left<\tau_2\tau_0\right>, \left<\tau_1^2\right>, \left<\tau_0^3\right>)$, and $(\left<\tau_4\right>, \left<\tau_2\tau_0\right>, \left<\tau_0^3\right>)$. 

\noindent We can also work with the restricted matroid $\mathcal{M}_k | S^k$. When we call the method .matroid\_repr(\ ), or the method .matrix\_repr(\ ), we have the option to set the parameter `positive' to `True', which restricts the matroid to $S^k$. For example,

\begin{verbatim}
sage: m12 = Descendent_Matroid(12)
sage: mat12 = m12.matroid_repr(positive = True); mat12
sage: Linear matroid of rank 7 on 9 elements represented over the Rational Field
sage: mat12.bases_count() == binomial(9, 7)
sage: True
\end{verbatim}

\noindent The .basis(\ ) method has an optional `positive' parameter, which, when set equal to True, returns an iterator containing the bases of the restricted matroid $\mathcal{M}_k | S^k$. \\

\subsection{The Descendent\_Basis(\ ) Class}

We also wrote a Descendent\_Basis(\ ) class that helps with computing the discriminant modular form in terms of the descendent bases of weight $k$ (see Corollary \ref{nonlineardiscriminant} and Table \ref{descendentbases}). When instantiating, the user specifies a weight $k$ and a type parameter (called `typ') ranging from $1$ to $8$. The type parameter specifies which of the eight triples of generators to work with. In Table \ref{descendentbases}, each row corresponds to one of these types. The first row of Table \ref{descendentbases} corresponds to setting the type parameter to $1$, the second row corresponds to setting the type parameter to $2$, and so on. For example, if we want to work with the triple $(\left<\tau_0\right>, \left<\tau_0^2\right>, \left<\tau_0^3\right>)$ i.e. type 1, in weight $4$, we first instantiate the basis,

\begin{verbatim}
sage: basis = Descendent_Basis(weight = 4, typ = 1); basis
sage: The descendent basis in weight 4 and type 1
\end{verbatim}

\noindent There is a method .expand(), which takes a parameter `coef'. The parameter `coef' is expected to be a list of $\text{dim}(\text{QM}_k)$ coefficients, say $[a_1, \ldots, a_{\text{dim}(\text{QM}_k)}]$. It then returns the quasimodular form of weight $k$, written in the descendent basis, whose first $\text{dim}(\text{QM}_k)$ coefficients are $a_1, \ldots, a_{\text{dim}(\text{QM}_k)}$. For example, suppose we want to know the expansion of $E_4$ in type 1. Since $\text{dim}(\text{QM}_4) = 2$, we need to specify the first $2$ coefficients of $E_4$, which happen to be $\frac{1}{240}$ and $1$. We have

\begin{verbatim}
sage: basis.expand([1/240, 1])
sage: {((0, 0),): 6/5, ((0,), (0,)): 6/5}
\end{verbatim}

\noindent We interpret the outputted dictionary as saying

\begin{equation*}
E_4 = \frac{6}{5}\left<\tau_0^2\right> + \frac{6}{5}\left<\tau_0\right>^2
\end{equation*}

\noindent Consider the case of weight $12$ and type 1. Suppose we want to find the discriminant modular form $\Delta$ in this descendent basis. Since $\text{dim}(\text{QM}_{12}) = 7$, we need to specify the first 7 coefficients of $\Delta$, which happen to be $0, 1, -24, 252, -1472, 4830, -6048$. We have

\begin{verbatim}
sage: basis = Descendent_Basis(weight = 12, typ = 1)
sage: basis.expand([0, 1, -24, 252, -1472, 4830, -6048]) 
sage: {((0, 0, 0), (0, 0, 0)): -432,
 ((0, 0, 0), (0, 0), (0,)): -7776,
 ((0, 0, 0), (0,), (0,), (0,)): -5184,
 ((0, 0), (0, 0), (0, 0)): 13824,
 ((0, 0), (0, 0), (0,), (0,)): 6480,
 ((0, 0), (0,), (0,), (0,), (0,)): -5184,
 ((0,), (0,), (0,), (0,), (0,), (0,)): -1728}
\end{verbatim}

\noindent The output is interpreted as

\begin{align*}
\Delta = & -432\left<\tau_0^3\right>^2 - 7776\left<\tau_0^3\right>\left<\tau_0^2\right>\left<\tau_0\right>  - 5184\left<\tau_0^3\right>\left<\tau_0\right>^3 \\
& + 13824\left<\tau_0^2\right>^3 + 6480\left<\tau_0^2\right>^2\left<\tau_0\right>^2 - 5184\left<\tau_0^2\right>\left<\tau_0\right>^4 - 1728\left<\tau_0\right>^6
\end{align*}

\noindent which corresponds to the first row of Table \ref{descendentbases}. The remaining rows are generated by varying the type parameter. 

\pagebreak

\subsection{Tables}

In this section, we collect various tables of expansions of $\Delta$ in terms of stationary descendents. In Table \ref{descendentbases}, we express $\Delta$ in the $8$ different descendent bases of weight $12$ (see Corollary \ref{nonlineardiscriminant}). In Tables \ref{partone}, \ref{parttwo}, and \ref{partthree}, we express $\Delta$ as a linear combination of stationary descendents of weight $12$. For Tables \ref{partone}, \ref{parttwo}, and \ref{partthree}, the first column corresponds to the $36$ bases of $\mathcal{M}_{12} | S^{12}$. The (ordered) groundset of this matroid is

\begin{equation*}
E^{12} | S^{12} = \left\{ \left<\tau_{10}\right>, \left<\tau_7\tau_1\right>, \left<\tau_6\tau_2\right>, \left<\tau_5\tau_3\right>, \left<\tau_4\tau_4\right>, \left<\tau_4\tau_1\tau_1\right>, \left<\tau_3\tau_2\tau_1\right>, \left<\tau_2\tau_2\tau_2\right>, \left<\tau_1\tau_1\tau_1\tau_1\right>  \right\}
\end{equation*} 

\noindent Each entry in the first column of Tables \ref{partone}, \ref{parttwo}, and \ref{partthree} is a string of seven distinct integers ranging from $1$ to $9$. Each integer keeps track of a stationary descendent in $E^{12} | S^{12}$. For example, the entry $(2456789)$ corresponds to the basis 

\begin{equation*}
 (\left<\tau_7\tau_1\right>, \left<\tau_5\tau_3\right>, \left<\tau_4\tau_4\right>, \left<\tau_4\tau_1\tau_1\right>, \left<\tau_3\tau_2\tau_1\right>, \left<\tau_2\tau_2\tau_2\right>, \left<\tau_1\tau_1\tau_1\tau_1\right>)
\end{equation*}

\noindent Here is some sample code to generate the first row of Table \ref{partone}:

\begin{verbatim}
sage: m12 = Descendent_Matroid(12)
sage: basis = list(m12.bases(positive = True))[0]
sage: target_vector = vector([0, 1, -24, 252, -1472, 4830, -6048])
sage: M = []
sage: for basis_vector in basis:
          f = Stationary_Descendent(basis_vector).expand(6)
          column = [i for i in f]
          M.append(column)
sage: M = Matrix(M).transpose()
sage: solution_vector = M.solve_right(target_vector)
sage: print(basis)
sage: [[10], [7, 1], [6, 2], [5, 3], [4, 4], [4, 1, 1], [3, 2, 1]]
sage: print(solution_vector)
sage: (-23011579448/8209, 90651811166/8209, -84309768312/8209, -83720227146/8209, 
      93735530480/8209, -944663370/8209, 550191978/8209)
\end{verbatim}

\noindent The results of the print statements are interpreted as

\begin{align*}
\Delta = & -\frac{23011579448}{8209}\left<\tau_{10}\right> + \frac{90651811166}{8209}\left<\tau_7\tau_1\right> - \frac{84309768312}{8209}\left<\tau_6\tau_2\right> - \frac{83720227146}{8209}\left<\tau_5\tau_3\right> \\
& + \frac{93735530480}{8209}\left<\tau_4\tau_4\right> - \frac{944663370}{8209}\left<\tau_4\tau_1\tau_1\right> + \frac{550191978}{8209}\left<\tau_3\tau_2\tau_1\right>
\end{align*}

\noindent which, after multiplying through by $8209$, is precisely the statement in the first row of Table \ref{partone}.

\begin{center}
\begin{table}[h]
\begin{adjustwidth}{-2.0cm}{}
\begin{tabular}{|c|l|}

\hline 

$\text{QM}$ in Descendent Generators & Discriminant Modular Form in Descendent Basis \\

\hline

$\mathbb{Q}\left[  \left<\tau_0\right>, \left<\tau_0^2\right>, \left<\tau_0^3\right> \right]$ & $\begin{aligned} \Delta = \ & 13824\left<\tau_0^2\right>^3 + 6480\left<\tau_0^2\right>^2\left<\tau_0\right>^2 \\ & - \left( 432\left<\tau_0^3\right>^2 + 7776\left<\tau_0^3\right>\left<\tau_0^2\right>\left<\tau_0\right> + 5184\left<\tau_0^3\right>\left<\tau_0\right>^3  + 5184\left<\tau_0^2\right>\left<\tau_0\right>^4 + 1728\left<\tau_0\right>^6\right) \end{aligned}$ \\

\hline

$\mathbb{Q}\left[  \left<\tau_0\right>, \left<\tau_0^2\right>, \left<\tau_1^2\right> \right]$ & $\begin{aligned} \Delta = \ & 155520\left<\tau_1^2\right>\left<\tau_0^2\right>\left<\tau_0\right> + 13824\left<\tau_0^2\right>^3 + 331776\left<\tau_0^2\right>\left<\tau_0\right>^4 \\ & - \left( 97200\left<\tau_1^2\right>^2 + 362880\left<\tau_1^2\right>\left<\tau_0\right>^3 + 20736\left<\tau_0^2\right>^2\left<\tau_0\right>^2 + 324864\left<\tau_0\right>^6  \right)   \end{aligned}$ \\

\hline 

$\mathbb{Q}\left[  \left<\tau_0\right>, \left<\tau_0^2\right>, \left<\tau_2\tau_0\right> \right]$ & $\begin{aligned} \Delta = \ & 25920\left<\tau_2\tau_0\right>\left<\tau_0^2\right>\left<\tau_0\right> + 13824\left<\tau_0^2\right>^3 + 37584\left<\tau_0^2\right>^2\left<\tau_0\right>^2 + 72576\left<\tau_0^2\right>\left<\tau_0\right>^4 \\
& - \left(43200\left<\tau_2\tau_0\right>^2 + 103680\left<\tau_2\tau_0\right>\left<\tau_0\right>^3 + 48384\left<\tau_0\right>^6   \right)   \end{aligned}$ \\

\hline

$\mathbb{Q}\left[  \left<\tau_0\right>, \left<\tau_0^2\right>, \left<\tau_4\right> \right]$ & $\begin{aligned} \Delta = \ & 3810240\left<\tau_4\right>\left<\tau_0^2\right>\left<\tau_0\right> + 10160640\left<\tau_4\right>\left<\tau_0\right>^3 + 13824\left<\tau_0^2\right>^3 \\ & - \left( 19051200\left<\tau_4\right>^2 + 149040\left<\tau_0^2\right>^2\left<\tau_0\right>^2 + 974592\left<\tau_0^2\right>\left<\tau_0\right>^4 + 1340928\left<\tau_0\right>^6  \right)   \end{aligned}$ \\

\hline

$\mathbb{Q}\left[  \left<\tau_0\right>, \left<\tau_2\right>, \left<\tau_0^3\right> \right]$ & $\begin{aligned} \Delta = \ & 41472\left<\tau_0^3\right>\left<\tau_0\right>^3 + 13824000\left<\tau_2\right>^3 + 14100480\left<\tau_2\right>\left<\tau_0\right>^4 \\ & - \left( 432\left<\tau_0^3\right>^2 + 77760\left<\tau_0^3\right>\left<\tau_2\right>\left<\tau_0\right> + 24235200\left<\tau_2\right>^2\left<\tau_0\right>^2 + 2723328\left<\tau_0\right>^6  \right)   \end{aligned}$ \\

\hline 

$\mathbb{Q}\left[  \left<\tau_0\right>, \left<\tau_2\right>, \left<\tau_1^2\right> \right]$ & $\begin{aligned} \Delta = \ & 1555200\left<\tau_1^2\right>\left<\tau_2\right>\left<\tau_0\right> + 13824000\left<\tau_2\right>^3 + 20736000\left<\tau_2\right>\left<\tau_0\right>^4 \\ & - \left( 97200\left<\tau_1^2\right>^2 + 1296000\left<\tau_1^2\right>\left<\tau_0\right>^3 + 26956800\left<\tau_2\right>^2\left<\tau_0\right>^2 + 604800\left<\tau_0\right>^6  \right)   \end{aligned}$ \\

\hline 

$\mathbb{Q}\left[  \left<\tau_0\right>, \left<\tau_2\right>, \left<\tau_2\tau_0\right> \right]$ & $\begin{aligned} \Delta = \ & 259200\left<\tau_2\tau_0\right>\left<\tau_2\right>\left<\tau_0\right> + 13824000\left<\tau_2\right>^3 + 11145600\left<\tau_2\right>\left<\tau_0\right>^4  \\ & - \left( 43200\left<\tau_2\tau_0\right>^2 + 259200\left<\tau_2\tau_0\right>\left<\tau_0\right>^3 + 21124800\left<\tau_2\right>^2\left<\tau_0\right>^2 + 2116800\left<\tau_0\right>^6  \right)   \end{aligned}$ \\

\hline 

$\mathbb{Q}\left[  \left<\tau_0\right>, \left<\tau_2\right>, \left<\tau_4\right> \right]$ & $\begin{aligned} \Delta = \ & 38102400\left<\tau_4\right>\left<\tau_2\right>\left<\tau_0\right> + 13824000\left<\tau_2\right>^3 + 23068800\left<\tau_2\right>\left<\tau_0\right>^4 \\ & - \left( 19051200\left<\tau_4\right>^2 + 12700800\left<\tau_4\right>\left<\tau_0\right>^3 + 39787200\left<\tau_2\right>^2\left<\tau_0\right>^2 + 3844800\left<\tau_0\right>^6  \right)   \end{aligned}$ \\

\hline

\end{tabular}
\end{adjustwidth}
\caption{This table collects the representations of the discriminant modular form in the $8$ descendent bases of weight $12$}
\label{descendentbases}
\end{table}
\end{center}

\pagebreak

\begin{table}[h]
\begin{adjustwidth}{-1.8cm}{}
\begin{tabular}{|c|c|c|}
\hline
Bases & Stationary Descendent Sum & Scaled $\Delta$ \\

\hline

(1234567) & $\begin{array}{c} -23011579448\left<\tau_{10}\right> + 90651811166\left<\tau_7\tau_1\right> -84309768312\left<\tau_6\tau_2\right> \\ -83720227146\left<\tau_5\tau_3\right> + 93735530480\left<\tau_4\tau_4\right> -944663370\left<\tau_4\tau_1\tau_1\right> \\ + 550191978\left<\tau_3\tau_2\tau_1\right>  \end{array}$ & $8209 \Delta$  \\

\hline

(1234578) & $\begin{array}{c} -1460661202000\left<\tau_{10}\right> -113056294400\left<\tau_7\tau_1\right> + 559690439952\left<\tau_6\tau_2\right> \\ -550617228000\left<\tau_5\tau_3\right> + 210717955840\left<\tau_4\tau_4\right>  + 1881300288\left<\tau_3\tau_2\tau_1\right> \\ -1511461392\left<\tau_2\tau_2\tau_2\right> \end{array}$ & $24447\Delta$ \\

\hline

(1235678) & $\begin{array}{c} 4061310267014000\left<\tau_{10}\right> + 2200026202847200\left<\tau_7\tau_1\right> -3456039328583856\left<\tau_6\tau_2\right> \\ + 1258635125332480\left<\tau_4\tau_4\right> -19271602980000\left<\tau_4\tau_1\tau_1\right> + 5388671207136\left<\tau_3\tau_2\tau_1\right> \\ + 4688332720176\left<\tau_2\tau_2\tau_2\right> \end{array}$ & $91636659\Delta$ \\

\hline 

(1345678) & $\begin{array}{c} -2143830146000400\left<\tau_{10}\right> + 654283897239504\left<\tau_6\tau_2\right> -942868372648800\left<\tau_5\tau_3\right> \\ + 471586339879680\left<\tau_4\tau_4\right> -1695844416000\left<\tau_4\tau_1\tau_1\right>  + 3695696563776\left<\tau_3\tau_2\tau_1\right> \\ -2175643467984\left<\tau_2\tau_2\tau_2\right> \end{array}$ & $49926419\Delta$ \\

\hline

(2345678) & $\begin{array}{c} 26235683604900\left<\tau_7\tau_1\right> -26432758699872\left<\tau_6\tau_2\right> -21300578323500\left<\tau_5\tau_3\right> \\ + 25663162697760\left<\tau_4\tau_4\right> -268128367500\left<\tau_4\tau_1\tau_1\right> + 147751230732\left<\tau_3\tau_2\tau_1\right> \\ + 6758645712\left<\tau_2\tau_2\tau_2\right> \end{array}$ & $2220683\Delta$ \\

\hline

(1245678) & $\begin{array}{c} -10587911748128\left<\tau_{10}\right> + 3207274006076\left<\tau_7\tau_1\right> -7260586824756\left<\tau_5\tau_3\right> \\ + 5466346730720\left<\tau_4\tau_4\right> -41153708820\left<\tau_4\tau_1\tau_1\right> + 36314615316\left<\tau_3\tau_2\tau_1\right> \\ -9918796272\left<\tau_2\tau_2\tau_2\right>  \end{array}$ & $518051\Delta$ \\

\hline

(1234678) & $\begin{array}{c} -366983226577968\left<\tau_{10}\right> -82527609478944\left<\tau_7\tau_1\right> + 195148578286704\left<\tau_6\tau_2\right> \\ -94397634399936\left<\tau_5\tau_3\right> +  553134634080\left<\tau_4\tau_1\tau_1\right> + 167863551072\left<\tau_3\tau_2\tau_1\right> \\ -393689228016\left<\tau_2\tau_2\tau_2\right>  \end{array}$ & $1561025\Delta$ \\

\hline

(1234568) & $\begin{array}{c} 4024462094224\left<\tau_{10}\right> + 1231898854592\left<\tau_7\tau_1\right> -2469393841488\left<\tau_6\tau_2\right> \\ + 769810172448\left<\tau_5\tau_3\right> + 319740097280\left<\tau_4\tau_4\right> -9406501440\left<\tau_4\tau_1\tau_1\right> \\ + 4401535824\left<\tau_2\tau_2\tau_2\right>  \end{array}$ & $10549\Delta$ \\

\hline

(1234589) & $\begin{array}{c} -205699794540240\left<\tau_{10}\right> +  212695958529600\left<\tau_7\tau_1\right> + 154182893788560\left<\tau_6\tau_2\right> \\ -79777455216480\left<\tau_5\tau_3\right> -43749966086400\left<\tau_4\tau_4\right> -86195237520\left<\tau_2\tau_2\tau_2\right> \\ + 15677502400\left<\tau_1\tau_1\tau_1\tau_2\right>  \end{array}$ & $8557351\Delta$ \\

\hline

(1235689) & $\begin{array}{c} 605422142631306000\left<\tau_{10}\right> + 974907030645589440\left<\tau_7\tau_1\right> -291379733109785040\left<\tau_6\tau_2\right> \\ -30403345965619200\left<\tau_4\tau_4\right> -2792210932576800\left<\tau_4\tau_1\tau_1\right> + 1059653000469840\left<\tau_2\tau_2\tau_2\right> \\ + 44905593392800\left<\tau_1\tau_1\tau_1\tau_1\right>  \end{array}$ & $27642454477\Delta$ \\

\hline

(1345689) & $\begin{array}{c} -221134586618877360\left<\tau_{10}\right> + 142554558857496240\left<\tau_6\tau_2\right> -52227162356013720\left<\tau_5\tau_3\right> \\ -24298943733057600\left<\tau_4\tau_4\right> + 398804922243000\left<\tau_4\tau_1\tau_1\right> -207776378545680\left<\tau_2\tau_2\tau_2\right> \\ + 3849683920600\left<\tau_1\tau_1\tau_1\tau_1\right>  \end{array}$ & $1654054369\Delta$ \\

\hline

(2345689) & $\begin{array}{c} 5412384987175320\left<\tau_7\tau_1\right> +  549096100019880\left<\tau_6\tau_2\right> -793822739464125\left<\tau_5\tau_3\right> \\ -538119984372600\left<\tau_4\tau_4\right> -9439894417275\left<\tau_4\tau_1\tau_1\right> + 2724787540020\left<\tau_2\tau_2\tau_2\right> \\ + 307815064025\left<\tau_1\tau_1\tau_1\tau_1\right> \end{array}$ & $178603106\Delta$ \\

\hline

(1245689) & $\begin{array}{c} 7478165933604080\left<\tau_{10}\right> + 47518186285832080\left<\tau_7\tau_1\right> -5203209519817590\left<\tau_5\tau_3\right> \\ -3902715832691600\left<\tau_4\tau_4\right> -96364308617850\left<\tau_4\tau_1\tau_1\right> +  30948776806920\left<\tau_2\tau_2\tau_2\right> \\ + 2572285251550\left<\tau_1\tau_1\tau_1\tau_1\right> \end{array}$ & $1512115407\Delta$ \\

\hline

(1234689) & $\begin{array}{c} 30780463106112720\left<\tau_{10}\right> + 34018521226280640\left<\tau_7\tau_1\right> -16391406497304720\left<\tau_6\tau_2\right> \\ + 2280250947421440\left<\tau_5\tau_3\right> -114843660976800\left<\tau_4\tau_1\tau_1\right> +  46047225680400\left<\tau_2\tau_2\tau_2\right> \\ + 1398862925600\left<\tau_1\tau_1\tau_1\tau_1\right> \end{array}$ & $892342649\Delta$ \\

\hline

\end{tabular}
\end{adjustwidth}
\caption{Part $1$ of $3$}
\label{partone}
\end{table}

\pagebreak

\begin{table}[h]
\begin{adjustwidth}{-2.1cm}{}
\begin{tabular}{|c|c|c|}   
\hline
Bases & Stationary Descendent Sum & Scaled $\Delta$ \\
  
\hline

(1236789) & $\begin{array}{c} 880302160231494000\left<\tau_{10}\right> + 1286369729009184960\left<\tau_7\tau_1\right> -469055398297145520\left<\tau_6\tau_2\right> \\ -4076305720351200\left<\tau_4\tau_1\tau_1\right> + 162875067672960\left<\tau_3\tau_2\tau_1\right> + 1467620115823920\left<\tau_2\tau_2\tau_2\right> \\ + 56189068095200\left<\tau_1\tau_1\tau_1\tau_1\right> \end{array}$ & $37357962683\Delta$ \\

\hline

(1346789) & $\begin{array}{c} -8462645983875042960\left<\tau_{10}\right> + 4498485674608848720\left<\tau_6\tau_2\right> -2572739458018369920\left<\tau_5\tau_3\right> \\ + 7947823801428000\left<\tau_4\tau_1\tau_1\right> + 4859788746611520\left<\tau_3\tau_2\tau_1\right> -8163567407678160\left<\tau_2\tau_2\tau_2\right> \\ + 98247154141600\left<\tau_1\tau_1\tau_1\tau_1\right>\end{array}$ & $107865386539\Delta$ \\
 
\hline

(2346789) & $\begin{array}{c} 301935421145820\left<\tau_7\tau_1\right> -261358755840\left<\tau_6\tau_2\right> -62815909821000\left<\tau_5\tau_3\right> \\ -762732650400\left<\tau_4\tau_1\tau_1\right> + 156886292820\left<\tau_3\tau_2\tau_1\right> + 145156976640\left<\tau_2\tau_2\tau_2\right> \\ + 15587440900\left<\tau_1\tau_1\tau_1\tau_1\right>\end{array}$ & $11402261\Delta$ \\

\hline

(1246789) & $\begin{array}{c} -1831340802170880\left<\tau_{10}\right> + 1124621418652212180\left<\tau_7\tau_1\right> -234527699148572760\left<\tau_5\tau_3\right> \\ -2839236826725600\left<\tau_4\tau_1\tau_1\right> + 585407374903740\left<\tau_3\tau_2\tau_1\right> + 538900800122880\left<\tau_2\tau_2\tau_2\right> \\ + 58079934013900\left<\tau_1\tau_1\tau_1\tau_1\right> \end{array}$ & $42493440431\Delta$ \\

\hline

(1456789) & $\begin{array}{c}-6220670014586536080\left<\tau_{10}\right> -3874730354453327910\left<\tau_5\tau_3\right> + 3213204053292034800\left<\tau_4\tau_4\right> \\ -19431171313622250\left<\tau_4\tau_1\tau_1\right> + 20364936979642320\left<\tau_3\tau_2\tau_1\right> -6733822630660920\left<\tau_2\tau_2\tau_2\right> \\ -97363675184450\left<\tau_1\tau_1\tau_1\tau_1\right> \end{array}$ & $233283588587\Delta$ \\

\hline

(2456789) & $\begin{array}{c} 1436360011284660\left<\tau_7\tau_1\right> -298080526863000\left<\tau_5\tau_3\right> -1208167833600\left<\tau_4\tau_4\right> \\ -3618950815200\left<\tau_4\tau_1\tau_1\right> + 740021698140\left<\tau_3\tau_2\tau_1\right> + 690812928000\left<\tau_2\tau_2\tau_2\right> \\ + 74215966700\left<\tau_1\tau_1\tau_1\tau_1\right> \end{array}$ & $54184663\Delta$ \\

\hline

(3456789) & $\begin{array}{c} -2768257476294072\left<\tau_6\tau_2\right> -1660573736306625\left<\tau_5\tau_3\right> + 2689970115662760\left<\tau_4\tau_4\right> \\ -21157920186375\left<\tau_4\tau_1\tau_1\right> + 14058142823832\left<\tau_3\tau_2\tau_1\right> -613637118588\left<\tau_2\tau_2\tau_2\right> \\ -141967984875\left<\tau_1\tau_1\tau_1\tau_1\right> \end{array}$ & $128918288\Delta$ \\

\hline 

(1356789) & $\begin{array}{c} 41793319795365138000\left<\tau_{10}\right> -43396979969877272592\left<\tau_6\tau_2\right> + 20581915664146959360\left<\tau_4\tau_4\right> \\ -201137590889172000\left<\tau_4\tau_1\tau_1\right> + 83563459769621952\left<\tau_3\tau_2\tau_1\right> + 35621145778898832\left<\tau_2\tau_2\tau_2\right> \\ -1571447287748000\left<\tau_1\tau_1\tau_1\tau_1\right> \end{array}$  & $453699429793\Delta$ \\

\hline

(2356789) & $\begin{array}{c} 2922609775899660\left<\tau_7\tau_1\right> + 1011089147119296\left<\tau_6\tau_2\right> -984953465993280\left<\tau_4\tau_4\right> \\ + 364200895800\left<\tau_4\tau_1\tau_1\right> -3628903951836\left<\tau_3\tau_2\tau_1\right> + 1629747391584\left<\tau_2\tau_2\tau_2\right> \\ + 202862650700\left<\tau_1\tau_1\tau_1\tau_1\right> \end{array}$ & $63164729\Delta$ \\

\hline

(1256789) & $\begin{array}{c} 6024406168252472000\left<\tau_{10}\right> + 18082074987448863580\left<\tau_7\tau_1\right> -3127035988647636800\left<\tau_4\tau_4\right> \\ -26740201199682600\left<\tau_4\tau_1\tau_1\right> -10406419039635180\left<\tau_3\tau_2\tau_1\right> + 15217887008292000\left<\tau_2\tau_2\tau_2\right> \\ + 1028583133507100\left<\tau_1\tau_1\tau_1\tau_1\right> \end{array}$ &  $456197441969\Delta$ \\

\hline

(1235789) & $\begin{array}{c} 57867475666000\left<\tau_{10}\right> + 2234862120990800\left<\tau_7\tau_1\right> + 713072031991536\left<\tau_6\tau_2\right> \\ -724676572506880\left<\tau_4\tau_4\right> -2659248507216\left<\tau_3\tau_2\tau_1\right> + 1295557142544\left<\tau_2\tau_2\tau_2\right> \\ + 152949230000\left<\tau_1\tau_1\tau_1\tau_1\right> \end{array}$ & $48929021\Delta$ \\

\hline

(1345789) & $\begin{array}{c} -3042871630003920\left<\tau_{10}\right> + 1243594964071824\left<\tau_6\tau_2\right> -1149357662223840\left<\tau_5\tau_3\right> \\ + 363329088065280\left<\tau_4\tau_4\right> + 3646216431936\left<\tau_3\tau_2\tau_1\right> -3018216215184\left<\tau_2\tau_2\tau_2\right> \\ + 16150899200\left<\tau_1\tau_1\tau_1\tau_1\right> \end{array}$ & $56197367\Delta$ \\

\hline

(2345789) & $\begin{array}{c} 186189697640100\left<\tau_7\tau_1\right> + 61377405825792\left<\tau_6\tau_2\right> -1821004479000\left<\tau_5\tau_3\right> \\ -59798239791360\left<\tau_4\tau_4\right> -215769015252\left<\tau_3\tau_2\tau_1\right> + 103152835968\left<\tau_2\tau_2\tau_2\right> \\ + 12768017500\left<\tau_1\tau_1\tau_1\tau_1\right> \end{array}$ & $4165387\Delta$ \\

\hline

(1245789) & $\begin{array}{c} 2089754531687680\left<\tau_{10}\right> + 2590822841816300\left<\tau_7\tau_1\right> + 764005748562360\left<\tau_5\tau_3\right> \\ -1081614029228800\left<\tau_4\tau_4\right> -5506531921020\left<\tau_3\tau_2\tau_1\right> + 3508189864320\left<\tau_2\tau_2\tau_2\right> \\ + 166574535700\left<\tau_1\tau_1\tau_1\tau_1\right> \end{array}$ & $19366513\Delta$ \\

\hline

\end{tabular}
\end{adjustwidth}
\caption{Part 2 of 3}   
\label{parttwo}
\end{table}

\pagebreak
 
\begin{table}[h]
\begin{adjustwidth}{-2.1cm}{}\label{a}
\begin{tabular}{|c|c|c|}
\hline
Bases & Stationary Descendent Sum & Scaled $\Delta$ \\

\hline

 (1234789) & $\begin{array}{c} -133611692033820\left<\tau_{10}\right> + 49673898758925\left<\tau_7\tau_1\right> + 70980920668140\left<\tau_6\tau_2\right> \\ -50953821504390\left<\tau_5\tau_3\right> + 102538983015\left<\tau_3\tau_2\tau_1\right> -105008740620\left<\tau_2\tau_2\tau_2\right> \\ + 4115585075\left<\tau_1\tau_1\tau_1\tau_1\right> \end{array}$ & $3578903\Delta$ \\

\hline

(1234579) & $\begin{array}{c} -24584759239040\left<\tau_{10}\right> + 44015653138100\left<\tau_7\tau_1\right> + 24557329050240\left<\tau_6\tau_2\right> \\ -9716678569080\left<\tau_5\tau_3\right> -11200932332800\left<\tau_4\tau_4\right> -21548809380\left<\tau_3\tau_2\tau_1\right> \\ + 3148877900\left<\tau_1\tau_1\tau_1\tau_1\right> \end{array}$ & $1438751\Delta$ \\

\hline

(1235679) & $\begin{array}{c} -67974047457316000\left<\tau_{10}\right> + 103895008521788260\left<\tau_7\tau_1\right> + 106525209058044000\left<\tau_6\tau_2\right> \\  -68488938738449600\left<\tau_4\tau_4\right> + 340083749917800\left<\tau_4\tau_1\tau_1\right> -264913250117460\left<\tau_3\tau_2\tau_1\right> \\ + 9767359833700\left<\tau_1\tau_1\tau_1\tau_1\right> \end{array}$ & $1507512793\Delta$ \\

\hline

(1345679) & $\begin{array}{c} 19305023750778480\left<\tau_{10}\right> -94273516829252880\left<\tau_6\tau_2\right> -44526432223623540\left<\tau_5\tau_3\right> \\ + 81635674076781600\left<\tau_4\tau_4\right> -660234797071500\left<\tau_4\tau_1\tau_1\right> + 415552757091360\left<\tau_3\tau_2\tau_1\right> \\ -4532590558300\left<\tau_1\tau_1\tau_1\tau_1\right> \end{array}$ & $3666371323\Delta$ \\

\hline

(2345679) & $\begin{array}{c} 18173099281260\left<\tau_7\tau_1\right> -39429476352000\left<\tau_6\tau_2\right> -27423633993000\left<\tau_5\tau_3\right> \\ + 38299109990400\left<\tau_4\tau_4\right> -347148967200\left<\tau_4\tau_1\tau_1\right> + 209599041540\left<\tau_3\tau_2\tau_1\right> \\ -1083116300\left<\tau_1\tau_1\tau_1\tau_1\right> \end{array}$ & $2521793\Delta$ \\

\hline

(1245679) & $\begin{array}{c} -3490447454208000\left<\tau_{10}\right> + 7856126402437740\left<\tau_7\tau_1\right> -3804471752073000\left<\tau_5\tau_3\right> \\  + 1796336000409600\left<\tau_4\tau_4\right> -30696661312800\left<\tau_4\tau_1\tau_1\right> +  15474388403460\left<\tau_3\tau_2\tau_1\right> \\ + 351290701300\left<\tau_1\tau_1\tau_1\tau_1\right> \end{array}$ & $427257857\Delta$ \\

\hline

(1234679) & $\begin{array}{c} -7119804547215360\left<\tau_{10}\right> + 14286242963436780\left<\tau_7\tau_1\right> + 3772305600860160\left<\tau_6\tau_2\right> \\ -5136670405383720\left<\tau_5\tau_3\right> -29402447373600\left<\tau_4\tau_1\tau_1\right> + 11511806420100\left<\tau_3\tau_2\tau_1\right> \\ + 820185891700\left<\tau_1\tau_1\tau_1\tau_1\right> \end{array}$ & $630253793\Delta$ \\

\hline

(1234569) & $\begin{array}{c} -74218022518640\left<\tau_{10}\right> + 138517585697120\left<\tau_7\tau_1\right> + 61897553613840\left<\tau_6\tau_2\right> \\ -37844750016780\left<\tau_5\tau_3\right> -21927250324000\left<\tau_4\tau_4\right> -107744046900\left<\tau_4\tau_1\tau_1\right> \\ + 9169866300\left<\tau_1\tau_1\tau_1\tau_1\right> \end{array}$ & $5126077\Delta$ \\

\hline

\end{tabular}
\end{adjustwidth}
\caption{Part 3 of 3}
\label{partthree}
\end{table}

\clearpage   
   
\bibliographystyle{alpha}
\bibliography{bibliography}

\newcommand{\etalchar}[1]{$^{#1}$}
\begin{thebibliography}{BvdGH{\etalchar{+}}08}

\bibitem[Afa]{Website}
Adam Afandi.
\newblock Authors website.
\newblock \url{https://sites.google.com/view/adamafandi}.

\bibitem[Ard21]{ardila2021geometry}
Federico Ardila.
\newblock The geometry of geometries: matroid theory, old and new.
\newblock {\em arXiv preprint arXiv:2111.08726}, 2021.

\bibitem[BO00]{bloch2000character}
Spencer Bloch and Andrei Okounkov.
\newblock The character of the infinite wedge representation.
\newblock {\em Advances in Mathematics}, 149(1):1--60, 2000.

\bibitem[BvdGH{\etalchar{+}}08]{zagier2008elliptic}
Jan~Hendrik Bruinier, Gerard van~der Geer, G{\"u}nter Harder, Don Zagier, and
  Don Zagier.
\newblock Elliptic modular forms and their applications.
\newblock {\em The 1-2-3 of modular forms: Lectures at a summer school in
  Nordfjordeid, Norway}, pages 1--103, 2008.

\bibitem[Del74]{deligne1974conjecture}
Pierre Deligne.
\newblock La conjecture de weil. i.
\newblock {\em Publications Math{\'e}matiques de l'Institut des Hautes
  {\'E}tudes Scientifiques}, 43:273--307, 1974.

\bibitem[DS05]{diamond2005first}
Fred Diamond and Jerry~Michael Shurman.
\newblock {\em A first course in modular forms}, volume 228.
\newblock Springer, 2005.

\bibitem[FH13]{fulton2013representation}
William Fulton and Joe Harris.
\newblock {\em Representation theory: a first course}, volume 129.
\newblock Springer Science \& Business Media, 2013.

\bibitem[GS18]{garvan2018}
Frank Garvan and Michael~J. Schlosser.
\newblock Combinatorial interpretations of ramanujan’s tau function.
\newblock {\em Discrete Mathematics}, 341(10):2831--2840, 2018.

\bibitem[KZ95]{kaneko1995generalized}
Masanobu Kaneko and Don Zagier.
\newblock A generalized jacobi theta function and quasimodular forms.
\newblock In {\em The moduli space of curves}, pages 165--172. Springer, 1995.

\bibitem[Leh47]{lehmer1947vanishing}
D.~H. Lehmer.
\newblock The vanishing of {Ramanujan}'s function {{\(\tau(n)\)}}.
\newblock {\em Duke Math. J.}, 14:429--433, 1947.

\bibitem[MM17]{mordell1917mr}
Louis~Joel Mordell and On~Mr.
\newblock On mr. ramanujan’s empirical expansions of modular functions.
\newblock In {\em Proc. Cambridge Philos. Soc}, volume~19, pages 117--124,
  1917.

\bibitem[Nie75]{niebur1975formula}
Douglas Niebur.
\newblock A formula for ramanujan's tau function.
\newblock {\em Illinois Journal of Mathematics}, 19(3):448--449, 1975.

\bibitem[OO97]{okounkov1997shifted}
Andrei Okounkov and Grigori Olshanski.
\newblock Shifted schur functions.
\newblock {\em Algebra i Analiz}, 9(2):73--146, 1997.

\bibitem[OP06]{okounkov2006gromov}
Andrei Okounkov and Rahul Pandharipande.
\newblock Gromov-witten theory, hurwitz theory, and completed cycles.
\newblock {\em Annals of mathematics}, pages 517--560, 2006.

\bibitem[Oxl11]{oxley2011matroid}
James Oxley.
\newblock {\em {Matroid Theory}}.
\newblock Oxford University Press, 02 2011.

\bibitem[Ram16]{ramanujan1916certain}
Srinivasa Ramanujan.
\newblock On certain arithmetical functions.
\newblock {\em Trans. Cambridge Philos. Soc}, 22(9):159--184, 1916.

\bibitem[Roy12]{royer2012quasimodular}
Emmanuel Royer.
\newblock Quasimodular forms: an introduction.
\newblock In {\em Annales math{\'e}matiques Blaise Pascal}, volume~19, pages
  297--306, 2012.

\bibitem[Sag13]{sagan2013symmetric}
Bruce~E Sagan.
\newblock {\em The symmetric group: representations, combinatorial algorithms,
  and symmetric functions}, volume 203.
\newblock Springer Science \& Business Media, 2013.

\bibitem[Ser12]{serre2012course}
Jean-Pierre Serre.
\newblock {\em A course in arithmetic}, volume~7.
\newblock Springer Science \& Business Media, 2012.

\bibitem[{The}]{sagemath}
{The Sage Developers}.
\newblock {\em {S}ageMath, the {S}age {M}athematics {S}oftware {S}ystem
  ({V}ersion 9.8)}.
\newblock {\tt https://www.sagemath.org}.

\bibitem[Whi92]{whitney1992abstract}
Hassler Whitney.
\newblock On the abstract properties of linear dependence.
\newblock {\em Hassler Whitney Collected Papers}, pages 147--171, 1992.

\end{thebibliography}

\end{document}